\documentclass[a4paper,UKenglish,cleveref, autoref, thm-restate,pdfa]{lipics-v2021}

\pdfoutput=1 
\hideLIPIcs  

\usepackage{proof}
\usepackage{tikz}
\usepackage{tikz-cd}
\usepackage{cmll}
\usetikzlibrary{patterns,arrows, topaths, calc, positioning}
\tikzset{>=stealth}
\tikzstyle{node} = [circle, minimum size = 1.1mm, inner sep = 0mm, draw={black}, fill]
\tikzstyle{hyperedge} = [rectangle, minimum width = 5mm, minimum height = 5mm, draw, inner sep = 0mm]

\newcommand{\eqdef}{:=}
\newcommand{\dom}{\mathrm{dom}}
\newcommand{\ran}{\mathrm{ran}}
\newcommand{\Nat}{\mathbb{N}}

\newcommand{\Gram}{\mathcal{G}}
\newcommand{\lab}{\mathit{lab}}
\newcommand{\att}{\mathit{att}}
\newcommand{\ext}{\mathit{ext}}
\newcommand{\type}{\mathrm{type}}
\newcommand{\HG}{\mathcal{H}}
\newcommand{\SG}{\mathrm{sg}}

\newcommand{\lt}{\mathbf{s}}
\newcommand{\rt}{\mathbf{t}}
\newcommand{\str}{{\mathit{str}}}

\newcommand{\Var}{\mathrm{Var}}
\newcommand{\FVar}{\mathrm{FVar}}

\newcommand{\Fm}{\mathrm{Fm}}

\newcommand{\mconj}{\otimes}
\newcommand{\mdisj}{\parr}
\newcommand{\bigmdisj}{\bigparr}
\newcommand{\limpl}{\multimap}
\newcommand{\aconj}{\wedge}
\newcommand{\adisj}{\vee}
\newcommand{\bang}{{!}}
\newcommand{\yields}{\vdash}

\newcommand{\Logic}{\mathcal{L}}
\newcommand{\LLFO}{\mathrm{LL}1}

\newcommand{\ILLFO}{\mathrm{ILL}1}
\newcommand{\MILLFO}{\mathrm{MILL}1}

\newcommand{\fm}{\mathit{fm}}
\newcommand{\diag}{\mathcal{D}}

\newcommand{\pc}{\mathbin{/\mkern-6mu/}}
\newcommand{\sub}{\mathit{sub}}

\title{First-Order Intuitionistic Linear Logic and Hypergraph Languages} 

\author{Tikhon Pshenitsyn}{Steklov Mathematical Institute of RAS, 8 Gubkina St., Moscow 119991, Russian Federation \and \url{https://www.researchgate.net/profile/Tikhon-Pshenitsyn} }{tpshenitsyn@mi-ras.ru}{https://orcid.org/0000-0003-4779-3143}{}

\authorrunning{T. Pshenitsyn} 

\Copyright{Tikhon Pshenitsyn} 



\funding{This work was supported by the Russian Science Foundation under grant no.~23-11-00104, \url{https://rscf.ru/en/project/23-11-00104/}.}

\acknowledgements{I thank Richard Moot and Sergei Slavnov for their attention to my work and for the productive discussions. I am also grateful to the anonymous reviewers for valuable remarks.}

\nolinenumbers 

\EventEditors{Keren Censor-Hillel, Fabrizio Grandoni, Joel Ouaknine, and Gabriele Puppis}
\EventNoEds{4}
\EventLongTitle{52nd International Colloquium on Automata, Languages, and Programming (ICALP 2025)}
\EventShortTitle{ICALP 2025}
\EventAcronym{ICALP}
\EventYear{2025}
\EventDate{July 8--11, 2025}
\EventLocation{Aarhus, Denmark}
\EventLogo{}
\SeriesVolume{334}
\ArticleNo{150}

\begin{document}

\maketitle

\begin{abstract}
	The Lambek calculus is a substructural logic known to be closely related to the formal language theory: on the one hand, it is used for generating formal languages by means of categorial grammars and, on the other hand, it has formal language semantics, with respect to which it is sound and complete. This paper studies a similar relation between first-order intuitionistic linear logic ILL1 along with its multiplicative fragment MILL1 on the one hand and the hypergraph grammar theory on the other. 
	In the first part, we introduce a novel concept of hypergraph first-order logic categorial grammar, which is a generalisation of string MILL1 grammars studied e.g.~in Richard Moot's 2014 works. We prove that hypergraph ILL1 grammars generate all recursively enumerable hypergraph languages and that hypergraph MILL1 grammars are as powerful as linear-time hypergraph transformation systems. In addition, we show that the class of languages generated by string MILL1 grammars is closed under intersection and that it includes a non-semilinear language as well as an NP-complete one. This shows how much more powerful string MILL1 grammars are as compared to Lambek categorial grammars.
	\\
	In the second part, we develop hypergraph language models for MILL1. In such models, formulae of the logic are interpreted as hypergraph languages and multiplicative conjunction is interpreted using parallel composition, which is one of the operations of HR-algebras introduced by Courcelle. We prove completeness of the universal-implicative fragment of MILL1 with respect to these models and thus present a new kind of semantics for a fragment of first-order linear logic. 
\end{abstract}

\section{Introduction}\label{section:introduction}

There is a strong connection between substructural logics, especially non-commutative ones, and the theory of formal languages and grammars \cite{Buszkowski03, MootR12}. This connection is two-way. On the one hand, a logic can be used as a derivational mechanism for generating formal languages, which is the essence of categorial grammars. One prominent example is Lambek categorial grammars based on the Lambek calculus $\mathrm{L}$ \cite{Lambek58}; formulae of the latter are built using multiplicative conjunction `$\cdot$' and two directed implications `$\backslash$', `$/$'. In a Lambek categorial grammar, one assigns a finite number of formulae of $\mathrm{L}$ to each symbol of an alphabet and chooses a distinguished formula $S$; then, the grammar accepts a string $a_1\ldots a_n$ if the sequent $A_1,\ldots,A_n \vdash S$ is derivable in $\mathrm{L}$ where, for $i=1,\ldots,n$, $A_i$ is one of the formulae assigned to $a_i$. A famous result by Pentus \cite{Pentus93} says that Lambek categorial grammars generate exactly context-free languages (without the empty word, to be precise).

On the other hand, algebras of formal languages can serve as models for substructural logics. For example, one can define language semantics for the Lambek calculus as follows: a language model is a function $u$ mapping formulas of $\mathrm{L}$ to formal languages such that $u(A \cdot B) = \{vw \mid v \in u(A), w \in u(B)\}$, $u(B \backslash A) = \{w \mid \forall v \in u(B)\; vw \in u(A)\}$, and $u(A/B) = \{v \mid \forall w \in u(B)\; vw \in u(A)\}$; a sequent $A \vdash B$ is interpreted as inclusion $u(A) \subseteq u(B)$. Another famous result by Pentus \cite{Pentus95} is that $\mathrm{L}$ is sound and complete w.r.t. language semantics; strong completeness for the fragment of $\mathrm{L}$ without `$\cdot$' had been proved earlier by Buszkowski in \cite{Buszkowski82} using canonical models.

Numerous variants and extensions of the Lambek calculus have been studied, including its nonassociative version \cite{MootR12}, commutative version \cite{vanBenthem95} (i.e. multiplicative intuitionistic linear logic), the multimodal Lambek calculus \cite{Moortgat96}, the displacement calculus \cite{MorrillVF10} etc. These logics have many common properties, which motivates searching for a unifying logic. One such ``umbrella'' logic is the first-order multiplicative intuitionistic linear logic $\MILLFO$ \cite{Moot14, MootP01}, which is the multiplicative fragment of first-order intuitionistic linear logic $\ILLFO$. The Lambek calculus can be embedded in $\MILLFO$ \cite{MootP01}: for example, the $\mathrm{L}$ sequent $p \cdot p \backslash q \vdash q$ is translated into the $\MILLFO$ sequent $p(x_0,x_1) \mconj \forall y (p(y,x_1) \limpl q(y,x_2)) \vdash q(x_0,x_2)$. In such a translation, variables ``fix'' the order of formulae. Although $\MILLFO$ is a first-order generalisation of $\mathrm{L}$, derivability problems in these logics have the same complexity, namely, they are NP-complete.
One can define $\MILLFO$ categorial grammars in the same manner as Lambek categorial grammars \cite{Moot14,Slavnov23}. The former generalise the latter, hence generate all context-free languages (see the definition of the latter in \cite{Kallmeyer10}). Moot proved in \cite{Moot14} that $\MILLFO$ grammars generate all multiple context-free languages, hence some non-context-free ones, e.g. $\{ww \mid w \in \Sigma^\ast\}$.

It turns out that the interplay between propositional substructural logics and formal grammars can be elevated fruitfully to that between first-order substructural logics and hypergraph grammars. This is the subject of this article. Hypergraph grammar approaches generate sets of hypergraphs; usually they are designed as generalisations of grammar formalisms for strings. For example, hyperedge replacement grammar \cite{DrewesKH97} is a formalism that extends context-free grammar. A rule of a hyperedge replacement grammar allows one to replace a hyperedge in a hypergraph by another hypergraph; see an example below.
\begin{figure}[!h]
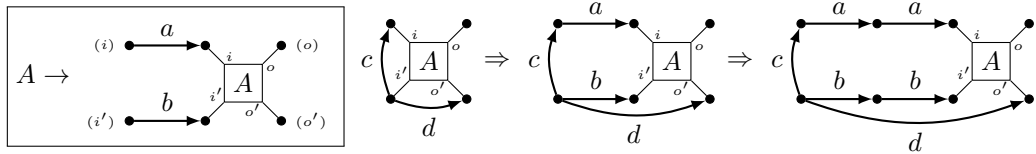

	$
	\boxed{
		A \to 
		\vcenter{\hbox{{\tikz[baseline=.1ex]{
						\def\HOR{1}
						\def\VER{1}
						\def\BEND{30}
						\node[hyperedge] (E1) at ($({\VER*1.5},\HOR*0.5)$) {$A$};
						\foreach \i in {0,...,2}
						{
							\foreach \j in {0,1}{
								\pgfmathtruncatemacro\NUM{3*\j+\i}
								\node[node, label=\ifnumequal{\i}{0}{left}{right}:{\tiny \ifnumequal{\NUM}{3}{$(i)$}{\ifnumequal{\NUM}{5}{$(o)$}{\ifnumequal{\NUM}{0}{$(i^\prime)$}{\ifnumequal{\NUM}{2}{$(o^\prime)$}{}}}}}] (V\i\j) at ($(\VER*\i,\HOR*\j)$) {};
							}
						}
						\draw[-] (E1) -- node[right] {\tiny $i$} (V11);
						\draw[-] (E1) -- node[below] {\tiny $o$} (V21);
						\draw[-] (E1) -- node[above] {\tiny $i^\prime$} (V10);
						\draw[-] (E1) -- node[left] {\tiny $o^\prime$} (V20);
						\draw[thick,-latex] (V01) -- node[above] {$a$} (V11);
						\draw[thick,-latex] (V00) -- node[above] {$b$} (V10);
	}}}}}
	\;
	\vcenter{\hbox{{\tikz[baseline=.1ex]{
					\def\HOR{1}
					\def\VER{1}
					\def\BEND{30}
					\node[hyperedge] (E1) at ($({\VER*0.5},\HOR*0.5)$) {$A$};
					\foreach \i in {0,1}
					{
						\foreach \j in {0,1}{
							\node[node] (V\i\j) at ($(\VER*\i,\HOR*\j)$) {};
						}
					}
					\draw[-] (E1) -- node[right] {\tiny $i$} (V01);
					\draw[-] (E1) -- node[below] {\tiny $o$} (V11);
					\draw[-] (E1) -- node[above] {\tiny $i^\prime$} (V00);
					\draw[-] (E1) -- node[left] {\tiny $o^\prime$} (V10);
					\draw[thick,-latex] (V00) to[bend left=20] node[left] {$c$} (V01);
					\draw[thick,-latex] (V00) to[bend right=20] node[below] {$d$} (V10);
					\def\SHIFT{2.2}
					\node[hyperedge] (2E1) at ($({\VER*1.5+\SHIFT},\HOR*0.5)$) {$A$};
					\foreach \i in {0,...,2}
					{
						\foreach \j in {0,1}{
							\node[node] (2V\i\j) at ($({\VER*\i+\SHIFT},\HOR*\j)$) {};
						}
					}
					\draw[-] (2E1) -- node[right] {\tiny $i$} (2V11);
					\draw[-] (2E1) -- node[below] {\tiny $o$} (2V21);
					\draw[-] (2E1) -- node[above] {\tiny $i^\prime$} (2V10);
					\draw[-] (2E1) -- node[left] {\tiny $o^\prime$} (2V20);
					\draw[thick,-latex] (2V01) -- node[above] {$a$} (2V11);
					\draw[thick,-latex] (2V00) -- node[above] {$b$} (2V10);
					\draw[thick,-latex] (2V00) to[bend left=20] node[left] {$c$} (2V01);
					\draw[thick,-latex] (2V00) to[bend right=20] node[below] {$d$} (2V20);
					\def\SSHIFT{5.4}
					\node[hyperedge] (3E1) at ($({\VER*2.5+\SSHIFT},\HOR*0.5)$) {$A$};
					\foreach \i in {0,...,3}
					{
						\foreach \j in {0,1}{
							\node[node] (3V\i\j) at ($({\VER*\i+\SSHIFT},\HOR*\j)$) {};
						}
					}
					\draw[-] (3E1) -- node[right] {\tiny $i$} (3V21);
					\draw[-] (3E1) -- node[below] {\tiny $o$} (3V31);
					\draw[-] (3E1) -- node[above] {\tiny $i^\prime$} (3V20);
					\draw[-] (3E1) -- node[left] {\tiny $o^\prime$} (3V30);
					\draw[thick,-latex] (3V01) -- node[above] {$a$} (3V11);
					\draw[thick,-latex] (3V00) -- node[above] {$b$} (3V10);
					\draw[thick,-latex] (3V11) -- node[above] {$a$} (3V21);
					\draw[thick,-latex] (3V10) -- node[above] {$b$} (3V20);
					\draw[thick,-latex] (3V00) to[bend left=20] node[left] {$c$} (3V01);
					\draw[thick,-latex] (3V00) to[bend right=20] node[below] {$d$} (3V30);
					\node at ($({(\SHIFT+\VER)/2-0.2},{\HOR*0.5})$) {$\Rightarrow$};
					\node at ($({(\SSHIFT+\SHIFT+2*\VER)/2-0.2},{\HOR*0.5})$) {$\Rightarrow$};
	}}}}
	$
	
	\caption{A hyperedge replacement rule (in a box) and an example of it being applied twice.}
	\label{figure:hrg}
\end{figure}

Note that, in this approach, hyperedges but not nodes are labeled. Some of the nodes, called external, are distinguished in a hypergraph; e.g., in the above example, these are the nodes marked by $(i),(o),(i^\prime),(o^\prime)$. External nodes are needed to specify how hyperedge replacement is done. A more general approach, which corresponds to type-0 grammars in the Chomsky hierarchy, is hypergraph transformation systems (the term is taken from \cite{Konig18}), which allow one to replace a subhypergraph in a hypergraph by another hypergraph.

Naturally, a hypergraph can be represented by a linear logic formula. Namely, one can interpret hyperedges as predicates, nodes as variables, and external nodes as free variables. For example, the hypergraph in the box from Figure \ref{figure:hrg} can be converted into the formula $\exists x.\exists y. a(i,x)\mconj b(i^\prime,y)\mconj A(x,o,y,o^\prime)$. This idea underlies the concept of \emph{hypergraph first-order categorial grammars}, which we introduce in Section \ref{section:hypergraph-categorial-grammars}. Roughly speaking, given a first-order logic $\Logic$, say, $\MILLFO$, a hypergraph $\Logic$ categorial grammar takes a hypergraph, assigns a formula of $\Logic$ to each its hyperedge and node, converts the resulting hypergraph into a $\Logic$ sequent and check whether it is derivable in $\Logic$. Using first-order linear logic for generating hypergraph languages is a novel idea which has not yet been explored in the literature. In Section \ref{section:hypergraph-categorial-grammars}, we study expressive power of grammars defined thusly and prove the following.
\begin{enumerate}
	\item \textit{Hypergraph $\ILLFO$ grammars are equivalent to hypergraph transformation systems and thus they generate all recursively enumerable hypergraph languages (Theorem \ref{theorem:ILL1G=RE}).} This result relates hypergraph $\ILLFO$ grammars to the well studied approach in the field of hypergraph grammars based on the double pushout graph transformation procedure.
	\item \emph{Hypergraph $\MILLFO$ grammars are at least as powerful as linear-time hypergraph transformation systems (Theorem \ref{theorem:MILL1G>LTHTS}).} The latter are hypergraph transformation systems where the length of a derivation is bounded by a linear function w.r.t. the size of the resulting hypergraph. The linear-time bound has been studied for many grammar formalisms \cite{Book71,Gladkii64,Pshenitsyn23,Tadaki10}, but, to our best knowledge, it is the first time it is used for graph grammars.
\end{enumerate}
The proofs partially use the techniques from \cite{Pshenitsyn23} where languages generated by grammars over the commutative Lambek calculus are studied. As compared to \cite{Pshenitsyn23}, the proofs in this paper are more technically involved because of complications arising when working with quantifiers and variables in the first-order setting.

In Section \ref{section:MILL1-grammars}, using the methods developed for hypergraph $\MILLFO$ grammars, we discover the following properties of the class of languages generated by string $\MILLFO$ grammars.
\begin{enumerate}
	\item \emph{Languages generated by string $\MILLFO$ grammars are closed under intersection (Theorem \ref{theorem:hypergraph-MILL1-grammars-intersection}). Consequently, non-semilinear languages, e.g. the language $\{a^{2n^2} \mid n \in \Nat\}$, can be generated by string $\MILLFO$ grammars.}
	\item \emph{String $\MILLFO$ grammars generate an NP-complete language (Theorem \ref{theorem:MILL1G-NPC}).} Note that the Lambek calculus is NP-complete \cite{Pentus06} but Lambek categorial grammars generate only context-free languages, which are in P. The logic $\MILLFO$ is NP-complete as well but string $\MILLFO$ grammars are able to generate non-polynomial languages (assuming P$\ne$NP), hence they are much more powerful.
\end{enumerate}
The question whether string $\MILLFO$ grammars generate only multiple context-free grammars was left open by Richard Moot in \cite{Moot14}, and I considered Theorem \ref{theorem:MILL1G-NPC} to be the first one answering it. Recently, however, Sergei Slavnov pointed out to an alternative answer to this question, using previously known techniques. Namely, in \cite{Moot14b}, it is shown that hybrid type-logical grammars can be translated into $\MILLFO$; hybrid type-logical grammars generalise abstract categorial grammars, and it is proved in \cite{Salvati10} that the latter generate an NP-complete language.
Our proof relies on a different technique, namely, on reducing linear-time hypergraph transformation systems, which are a rule-based approach unlike hybrid type-logical grammars. In general, finding a natural rule-based formalism equivalent to a $\MILLFO$ grammars is an interesting question to study, and Theorem \ref{theorem:MILL1G-NPC} is a step towards the answer\footnote{We believe that linear-time hypergraph transformation systems are essentially equivalent to $\MILLFO$ grammars in terms of generative power because, as we conjecture, it is possible to simulate $\MILLFO$ axioms and rules by hypergraph transformation rules and hence to convert hypergraph $\MILLFO$ grammars into linear-time transformation systems. However, some subtleties arise related to the definition of hypergraph transformation rules when one tries to do so; we discuss them after the proof of Theorem \ref{theorem:MILL1G>LTHTS}.}.


The second part of the paper (Section \ref{section:semantics}), not related to the first one, is devoted to developing hypergraph language semantics for $\MILLFO$, thus establishing the other way round connection between first-order linear logic and hypergraph languages. Linear logic is considered as a logic for reasoning about resources \cite{Girard93}, and language models for the Lambek calculus are one of formalisations of this statement with resources being words; this agrees with linguistic applications of $\mathrm{L}$. In this paper, we shall show that hypergraphs can be treated as ``first-order resources.'' In a hypergraph language model (Definition \ref{definition:model}), $\MILLFO$ formulas are interpreted by sets of hypergraphs, and the tensor operation $A \mconj B$ is interpreted using the parallel composition operation. The latter one is ``gluing'' of hypergraphs; it is studied well in the hypergraph grammar theory, namely, it is one of the operations in the language of HR-algebras \cite{CourcelleE12}. Hypergraph language models is a particular case of intuitionistic phase semantics (see the definition in \cite{KanovichOT06}) with the trivial closure operator $\mathrm{Cl}(X)=X$.

Our main result concerning hypergraph language models is soundness of $\MILLFO$ and completeness of its $\{\limpl,\forall\}$-fragment w.r.t. them (Theorem \ref{theorem:completeness}). The proof is inspired by Buszkowski's one \cite{Buszkowski82} but it is more technically involved, again because of the first-order setting. This result's importance amounts to the fact that hypergraph language models is one of few, to our best knowledge, examples of a specific semantics for a fragment of first-order intuitionistic linear logic, which, moreover, is grounded in the hypergraph language theory.


\section{Preliminaries}\label{section:preliminaries}

In Section \ref{subsection:hypergraphs}, we introduce notions from the field of graph grammars, and, in Section \ref{subsection:ILL1}, we introduce first-order intuitionistic linear logic. 

\subsection{Hypergraphs \& Hypergraph Transformation Systems}\label{subsection:hypergraphs}

There are many paradigms in the field of graph grammars, including node replacement grammars, hyperedge replacement grammars, algebraic approaches (double pushout, single pushout) with a more categori\textbf{c}al flavour, definability in monadic second-order logic etc. We shall work with the definition of a hypergraph from the field of hyperedge replacement grammars \cite{DrewesKH97,Engelfriet97,Habel92} because it fits first-order linear logic better than other ones. In hypergraphs we shall deal with, only hyperedges are labeled while nodes play an auxiliary role. Some of the nodes are marked as external; informally, they play the role of gluing points. Throughout the paper, we shall explore a natural correspondence between hypergraphs defined thusly and first-order linear logic formulae. In contrast, there would be no such correspondence if we sticked to the definition of a hypergraph where nodes are labeled and edges are not.

We fix a countable set $\Sigma$; its elements are called \emph{selectors}. In the grammar-logic correspondence we shall develop, selectors will guide variable substitution.
\begin{definition}
	A $\Sigma$-typed alphabet is a set $C$ along with a function $\type:C \to \mathcal{P}(\Sigma)$ such that $\type(c)$ is finite for $c \in C$.
\end{definition}

\begin{definition}
	Let $C$ be a finite $\Sigma$-typed alphabet of hyperedge labels. A \emph{hypergraph} over $C$ is a tuple $H = \langle V_H,E_H, \lab_H, \att_H, \ext_H\rangle$ where  $V_H$ is a finite set of nodes; $E_H$ is a finite set of hyperedges; for each $e \in E_H$, $\lab_H:E_H \to C$ is the labeling function and $\att_H(e):\type(\lab_H(e)) \to V_H$ is the attachment function; $\ext_H: \type(H) \to V_H$ is a function with the domain $\type(H) \subseteq \Sigma$. Elements of $\ran(\ext_H)$ are called \emph{external nodes}. The set of hypergraphs over $C$ is denoted by $\HG(C)$. 
	Let $\type_H(e) \eqdef \dom(\att_H(e))$ for each $e \in E_H$.
\end{definition}

In drawings of hypergraphs, nodes are depicted as black circles and hyperedges are depicted as labeled rectangles. When depicting a hypergraph $H$, we draw a line with a label $\sigma$ from $e$ to $v$ if $\att_H(e)(\sigma)=v$. External nodes are represented by symbols in round brackets: if $ext_H(\sigma)=v$, then we mark $v$ as $(\sigma)$. 

There is a standard issue with distinguishing between concrete and abstract hypergraphs, i.e. between hypergraphs and their isomorphism classes. When one considers a hypergraph language $L$, it is reasonable to assume that it consists of abstract hypergraphs (or, equivalently, that it is closed under isomorphism); however, when we write $H \in L$, we assume that $H$ is a concrete hypergraph. Following tradition, we often do not distinguish between abstract and concrete hypergraphs to avoid excessive bureaucracy.

Let us fix two selectors, $\lt$ and $\rt$. If $\type_H(e) = \{\lt,\rt\}$, then the hyperedge $e$ is called an \emph{edge} and it is depicted by an arrow going from $\att_H(e)(\lt)$ to $\att_H(e)(\rt)$.
\begin{definition}
	\emph{A string graph $\SG(w)$ induced by a string $w=a_1\dots a_n$} is defined as follows: $V_{\SG(w)} = \{v_0,\ldots,v_n\}$, $E_{\SG(w)} = \{e_1,\ldots,e_n\}$; $\type(e_i) = \type(\SG(w)) = \{\lt,\rt\}$, $\att_{\SG(w)}(e_i)(\lt)=v_{i-1}$, $\att_{\SG(w)}(e_i)(\rt)=v_{i}$, $\lab_{\SG(w)}(e_i)=a_i$ (for $i = 1, \ldots, n$); $\ext_{\SG(w)}(\lt)=v_0$, $\ext_{\SG(w)}(\rt)=v_n$. For example, $\SG(ab) = \vcenter{\hbox{{\tikz[baseline=.1ex]{
					\foreach \i in {1,...,3}
					{
						\node[node, label=\ifnumequal{\i}{1}{left}{right}:{\tiny \ifnumequal{\i}{1}{$(\lt)$}{\ifnumequal{\i}{3}{$(\rt)$}{}}}] (V\i) at ($(0.8*\i-0.8,0)$) {};
					}
					\draw[-latex, thick] (V1) -- node[above] {$a$} (V2);
					\draw[-latex, thick] (V2) -- node[above] {$b$} (V3);
	}}}}$.
\end{definition}
\begin{definition}
	Given $a \in C$, $a^\bullet$ is a hypergraph such that $V_{a^\bullet} = \type(a)$; $E_{a^\bullet} = \{e\}$ with $\type_{a^\bullet}(e)=\type(a)$; $\att_{a^\bullet}(\sigma)=\ext_{a^\bullet}(\sigma) = \sigma$ for $\sigma \in \type(a)$. For example, if $\type(a)=\{x,y,z\}$, then $a^\bullet = \vcenter{\hbox{{\tikz[baseline=.1ex]{
					\node[node, label=left:{\tiny $(x)$}] (V1) at ($(0,0.8)$) {};
					\node[node, label=left:{\tiny $(y)$}] (V2) at ($(0,0.4)$) {};
					\node[node, label=left:{\tiny $(z)$}] (V3) at ($(0,0)$) {};
					\node[hyperedge] (E) at (1,0.4) {$a$};
					\draw[-] (V1) to[bend left=10] node[above] {\tiny $x$} (E);
					\draw[-] (V2) -- node[above] {\tiny $y$} (E);
					\draw[-] (V3) to[bend right=10] node[above] {\tiny $z$} (E);
	}}}}$.
\end{definition}
\begin{definition}\label{definition:disjoint-union}
	If $H_1,H_2$ are hypergraphs with $\type(H_1)\cap\type(H_2) = \emptyset$, then their \emph{disjoint union} is the hypergraph $H_1+H_2 \eqdef \langle V_{H_1} \sqcup V_{H_2}, E_{H_1} \sqcup E_{H_2}, \att,\lab,\ext \rangle$ where $\att=\att_{H_i}$, $\lab=\lab_{H_i}$ on $E_i$ for $i=1,2$, and $\ext$ is the union of functions $\ext_{H_1}$ and $\ext_{H_2}$. 
\end{definition}

Now, let us define the hyperedge replacement operation.
\begin{definition}
	Let $H$ be a hypergraph and let $R$ be a binary relation on $V_H$. Let $\equiv_R$ be the smallest equivalence relation on $V_H$ containing $R$. Then $H/R = H^\prime$ is the following hypergraph: $V_{H^\prime} = \{[v]_{\equiv_R} \mid v \in V_H\}$; $E_{H^\prime} = E_H$; $\lab_{H^\prime} = \lab_H$; $\att_{H^\prime}(e)(s) = [\att_{H}(e)(s)]_{\equiv_R}$; $\ext_{H^\prime}(s) = [\ext_{H}(s)]_{\equiv_R}$.
\end{definition}
\begin{definition}
	Let $H,K$ be two hypergraphs over $C$; let $e \in E_H$ be a hyperedge such that $\type(e) = \type(K)$. Then the \emph{replacement of $e$ by $K$ in $H$} (the result being denoted by $H[e/K]$) is defined as follows:
	\begin{enumerate}
		\item Remove $e$ from $H$ and add a disjoint copy of $K$. Formally, let $L$ be the hypergraph such that $V_L = V_H \sqcup V_K$, $E_L = (E_H \setminus \{e\}) \sqcup E_K$, $\lab_L$ is the restriction of $\lab_H \cup \lab_K$ to $E_L$, $\att_L$ is the restriction of $\att_H \cup \att_K$ to $E_L$, and $\ext_L = \ext_H$.
		\item Glue the nodes that are incident to $e$ in $H$ with the external nodes of $K$. Namely, let $H[e/K] \eqdef L/R$ where $R = \{(\att_H(e)(s),\ext_K(s)) \mid s \in \type(e)\}$.
	\end{enumerate}
\end{definition}

Using hyperedge replacement, we define \emph{hypergraph transformation system}, a formalism which shall be used to describe expressive power of hypergraph categorial grammars. It enables one to replace a subhypergraph in a hypergraph with another hypergraph. 
\begin{definition}\label{definition:ht-rule-system}
	A \emph{hypergraph transformation rule (ht-rule)} is of the form $r = (H \to H^\prime)$ where $H,H^\prime$ are hypergraphs such that $\type(H) = \type(H^\prime)$ and $\ext_H,\ext_{H^\prime}$ are injective.
	\\
	We say that $G$ is transformed into $G^\prime$ via $r$ and denote this by $G \Rightarrow_r G^\prime$ (also by $G \Rightarrow G^\prime$) if $G = K[e/H]$ and $G^\prime = K[e/H^\prime]$ for some $K$ and $e \in E_K$ such that $\att_K(e)$ is injective. 
	\\
	A \emph{hypergraph transformation system (ht-system)} is a tuple $\Gram = \langle N,T,P,S\rangle$ where $N,T$ are $\Sigma$-typed alphabets, $P$ is a finite set of hypergraph transformation rules and $S \in \mathcal{H}(N \cup T)$ is a start hypergraph such that $\ext_S$ is injective.
	The language $L(\Gram)$ generated by $\Gram$ consists of hypergraphs $H \in \HG(T)$ such that $S \Rightarrow^\ast H$.
\end{definition}
\begin{figure}[!h]
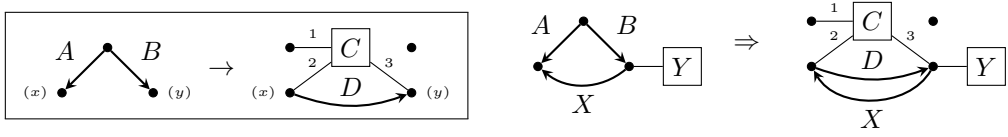

	$
		\boxed{
			\vcenter{\hbox{{\tikz[baseline=.1ex]{
							\node[node] (N) at (0,0.6) {};
							\node[node, label=left:{\tiny $(x)$}] (N1) at (-0.6,0) {};
							\node[node, label=right:{\tiny $(y)$}] (N2) at (0.6,0) {};
							\draw[->, thick] (N) -- node[above left] {$A$} (N1);
							\draw[->, thick] (N) -- node[above right] {$B$} (N2);
							\def\X{1.6}
							\def\Y{0.6}
							\def\SHIFT{2.4}
							\node at ($({(\SHIFT+0.6)/2},\Y/2)$) {$\to$};
							\node[node] (2N) at (\SHIFT,\Y) {};
							\node[node, label=left:{\tiny $(x)$}] at (\SHIFT,0) (2N1) {};
							\node[node, label=right:{\tiny $(y)$}] (2N2) at ($(\SHIFT+\X,0)$) {};
							\node[node] (2N3) at ($(\SHIFT+\X,\Y)$) {};
							\node[hyperedge] (2F) at ($(\SHIFT+\X/2,\Y)$)  {$C$};
							\draw[->, thick] (2N1) to[bend right=20] node[above] {$D$} (2N2);
							\draw[-] (2N) -- node[above] {\tiny 1} (2F);
							\draw[-] (2N1) -- node[above] {\tiny 2} (2F);
							\draw[-] (2N2) -- node[above] {\tiny 3} (2F);
		}}}}}
		\qquad
		\vcenter{\hbox{{\tikz[baseline=.1ex]{
						\def\X{1.6}
						\def\Y{0.6}
						\def\H{1.3}
						\def\SHIFT{3}
						\node[node] (N) at (0,0.6) {};
						\node[node] (N1) at (-0.6,0) {};
						\node[node] (N2) at (0.6,0) {};
						\node[hyperedge] (F2) at ($(\H,0)$)  {$Y$};
						\draw[->, thick] (N) -- node[above left] {$A$} (N1);
						\draw[->, thick] (N) -- node[above right] {$B$} (N2);
						\draw[->, thick] (N2) to[bend left=40] node[below] {$X$} (N1);
						\draw[-] (N2) -- (F2);
						\node at ($({(\SHIFT+\H)/2},\Y/2)$) {$\Rightarrow$};
						\node[node] (2N) at (\SHIFT,\Y) {};
						\node[node] at (\SHIFT,0) (2N1) {};
						\node[node] (2N2) at ($(\SHIFT+\X,0)$) {};
						\node[node] (2N3) at ($(\SHIFT+\X,\Y)$) {};
						\node[hyperedge] (2F) at ($(\SHIFT+\X/2,\Y)$)  {$C$};
						\node[hyperedge] (2F2) at ($(\SHIFT+\X+0.7,0)$)  {$Y$};
						\draw[->, thick] (2N1) to[bend right=20] node[above] {$D$} (2N2);
						\draw[<-, thick] (2N1) to[bend right=60] node[below] {$X$} (2N2);
						\draw[-] (2N2) -- (2F2);
						\draw[-] (2N) -- node[above] {\tiny 1} (2F);
						\draw[-] (2N1) -- node[above] {\tiny 2} (2F);
						\draw[-] (2N2) -- node[above] {\tiny 3} (2F);
		}}}}
	$
	
	\caption{An example of a hypergraph transformation rule (boxed) and of its application.}
\end{figure}

The definition of a hypergraph transformation rule, although given in a slightly unconventional way through hyperedge replacement, coincides with the standard notion of a graph transformation rule with injective morphisms of the double pushout approach formalism in the corresponding category of hypergraphs; compare it with \cite{Konig18}. Also, our definition is essentially the same as that from \cite{Uesu78} (with the only difference that the cited paper deals with graphs rather than with hypergraphs). In that paper, the following proposition is proved.
\begin{proposition}\label{proposition:hts=re}
	Hypergraph transformation systems generate all recursively enumerable hypergraph languages.
\end{proposition}
This is a rather expected result related to string rewriting systems generating all recursively enumerable string languages. To prove Proposition \ref{proposition:hts=re}, it suffices to show how to convert a string representation of a hypergraph into the hypergraph itself by means of ht-rules.

The injectivity requirements in Definition \ref{definition:ht-rule-system} are quite standard, ht-systems defined thusly correspond to the class of $\mathrm{DPO}^{i/i}$ grammars investigated in \cite{HabelMP01} (``DPO'' stands for ``double-pushout approach''). In this paper, injectivity is crucial in the proof of Theorem \ref{theorem:MILL1G>LTHTS}.

\subsection{First-Order Intuitionistic Linear Logic}\label{subsection:ILL1}

We assume the reader's familiarity with the basic principles and issues of first-order logic. Let us fix a countable set of variables $\Var$ and a countable set of predicate symbols with arities. Atomic formulae are of the form $p(x_1,\ldots,x_n)$ where $p$ is a predicate symbol of arity $n$ and $x_1,\ldots,x_n$ are variables. Following \cite{Komori86,Moot14,MootP01}, we do not allow function symbols; note that complex terms would not fit in the \emph{variables are nodes, predicates are hyperedges} paradigm. We also do not allow constants because they can easily be simulated by variables.

Formulae of intuitionistic linear logic $\ILLFO$ are built from atomic formulae and propositional constants $0,1,\top$ using the multiplicative connectives $\mconj,\limpl$, the additive ones $\aconj,\adisj$, and the exponential one $\bang$ along with the quantifiers $\exists,\forall$. The multiplicative fragment of $\ILLFO$ denoted by $\MILLFO$ does not have constants and uses only $\mconj,\limpl$ and $\exists,\forall$. A sequent is a structure of the form $\Gamma \vdash B$ where $\Gamma$ is a multiset of formulae and $B$ is a formula. 

Note that we shall sometimes describe multisets using the notation $\{f(x) \mid \Phi(x) \}$. An element $a$ belongs to this multiset $n$ times if there are exactly $n$ elements $x_1,\ldots,x_n$ such that $f(x_i)=a$ and such that $x_i$ satisfies $\Phi$.

The only axiom is $A \vdash A$. The rules for $\MILLFO$ are presented below.
$$
\infer[(\mconj L)]{\Gamma,A \mconj B \yields C}{\Gamma,A, B \yields C}
\quad
\infer[(\mconj R)]{\Gamma,\Delta \yields A \mconj B}{\Gamma \yields A & \Delta \yields B}
\quad
\infer[(\multimap L)]{\Gamma, \Pi, B \multimap A \yields C}{\Pi \yields B & \Gamma, A \yields C}
\quad
\infer[(\multimap R)]{\Gamma \yields B \multimap A}{\Gamma, B \yields A}
$$
$$
\infer[(\exists L)]{\Gamma, \exists x A \yields B}{\Gamma, A[z/x] \yields B} 
\qquad
\infer[(\exists R)]{\Gamma \yields \exists x A}{\Gamma \yields A[y/x]}
\qquad
\infer[(\forall L)]{\Gamma, \forall x A \yields B}{\Gamma, A[y/x] \yields B} 
\qquad
\infer[(\forall R)]{\Gamma \yields \forall x A}{\Gamma \yields A[z/x]}
$$
Here $y$ is any variable while $z$ is a variable which is not free in $\Gamma,A,B$; $A[y/x]$ denotes replacing all free occurrences of $x$ in $A$ by $y$. More generally, if $h:\FVar(A) \to \Var$ is a function defining a correct substitution, then $A[h]$ denotes the result of substituting $h(x)$ for $x$ for $x \in \FVar(A)$ ($\FVar(A)$ is the set of free variables in $A$). Rules for additive and exponential connectives of $\ILLFO$ can be found e.g.~in \cite[Appendix E]{Shellinx91}.

The cut rule is admissible in $\ILLFO$:
$$
\infer[(\mathrm{cut})]{\Gamma,\Delta \yields C}{\Gamma \yields A & A, \Delta \yields C}
$$
Using it, one can prove that the rules $(\mconj L)$, $(\limpl R)$, $(\exists L)$, and $(\forall R)$ are invertible, i.e.~that, if a conclusion of any of these rules is provable in $\ILLFO$, then so is its premise.

\section{Hypergraph First-Order Categorial Grammars}\label{section:hypergraph-categorial-grammars}

The idea of extending the Lambek calculus and categorial grammars to hypergraphs was explored recently in the work \cite{Pshenitsyn22}, where the hypergraph Lambek calculus was introduced. This is a propositional logic whose formulae are built using two operators, $\times$ and $\div$ (similar to linear logic $\mconj$ and $\limpl$). Formulae of this calculus can be used as labels on hyperedges: e.g. if $H$ is a hypergraph labeled by formulas, then $\times(H)$ is a formula. Based on the hypergraph Lambek calculus, hypergraph Lambek grammars were defined and their properties were investigated. Although the definition of the hypergraph Lambek calculus is justified in \cite{Pshenitsyn22}, the syntax of this logic is somewhat cumbersome. We are going to show that one can use any first-order logic, such as $\MILLFO$, as the underlying logic for hypergraph categorial grammars. The definitions we propose below are much simpler than those from \cite{Pshenitsyn22}; besides, they enable one to rely on the well studied apparatus of linear logic.

Let us start with the definition of a string categorial grammar over a first-order logic. This notion appears in \cite{Moot14,Slavnov23} for $\MILLFO$ but we would like to start with a more general exposition. Let $\Logic$ be a first-order sequent calculus of interest and let $\Fm(\Logic)$ denote the set of its formulas. Let $\lt,\rt$ be two fixed variables (note that earlier we used them as selectors).

\begin{definition}\label{definition:string-L-grammar}
	A string $\Logic$ grammar is a tuple $\Gram = \langle T, S, \triangleright \rangle$ where $T$ is a finite alphabet, $S$ is a formula of $\Logic$ such that $\FVar(S) \subseteq \{\lt,\rt\}$, and $\triangleright \subseteq T \times \Fm(\Logic)$ is a finite binary relation such that $a \triangleright A$ implies $\FVar(A) \subseteq \{\lt,\rt\}$.
	The language $L(\Gram)$ generated by $\Gram$ is defined as follows: $a_1\ldots a_n \in L(\Gram)$ if and only if there are formulas $A_1,\ldots,A_n$ such that $a_i \triangleright A_i$ for $i=1,\ldots,n$ and such that the sequent $A_1[x_0/\lt,x_1/\rt], \ldots, A_n[x_{n-1}/\lt,x_n/\rt] \vdash S[x_0/\lt,x_n/\rt]$ is derivable in $\Logic$ where $x_0,\ldots,x_n$ are distinct variables.
\end{definition}

\begin{example}\label{example:string-L-grammar}
	Let $T=\{a\}$, let $S = q(\lt,\rt)$, and let $\triangleright$ consist of the pairs $a \triangleright p(\lt,\rt)$, $a \triangleright \forall x. p(x,\lt) \limpl q(x,\rt)$. This grammar accepts the string $aa$, because the sequent $$p(x_0,x_1),\forall x. p(x,x_1) \limpl q(x,x_2) \vdash q(x_0,x_2)$$ is derivable in $\MILLFO$.
\end{example}

We see that, in Definition \ref{definition:string-L-grammar}, the noncommutative structure of a string is simulated by variables. Informally, one could imagine a string graph with the nodes $x_0,\ldots,x_n$ such that, for $i=1,\ldots,n$, there is an edge labeled by $A_i$ connecting $x_{i-1}$ to $x_i$. Based on this observation, let us introduce the central notion of hypergraph $\Logic$ grammars. From now on, we consider nodes, selectors and logical variables as objects of the same kind. Besides, if $T$ is a $\Sigma$-typed alphabet, then we treat $a \in T$ as a predicate symbol. 

\begin{definition}\label{definition:hypergraph-L-grammar}
	Let us fix a variable $x_\bullet$ and a symbol $\bullet$. A \emph{hypergraph $\Logic$ grammar} is a quadruple $\Gram = \langle T, S, X, \triangleright \rangle$ where $T$ is a $\Sigma$-typed alphabet; $\triangleright \subseteq (T \cup \{\bullet\}) \times \Fm({\Logic})$ is a finite binary relation such that $a \triangleright A$ implies $\FVar(A) \subseteq \type(a)$ and $\bullet \triangleright A$ implies $\FVar(A) \subseteq \{x_\bullet\}$; finally, $S$ is a formula of $\Logic$ such that $\FVar(S) \subseteq X \subseteq \Sigma$.
\end{definition}

\begin{definition}\label{definition:language-hypergraph-L-grammar}
	The language $L(\Gram)$ is defined as follows: $H \in L(\Gram)$ if and only if $\type(H)=X$ and there are functions $h_V:V_H \to \Fm({\Logic})$, $h_E:E_H \to \Fm({\Logic})$ such that 
	\begin{enumerate}
		\item $\bullet \triangleright h_V(v)$ for $v \in V_H$, $\lab_H(e) \triangleright h_E(e)$ for $e \in E_H$;
		\item the sequent $\{h_E(e)[\att_H(e)] \mid e \in E_H \},\{h_V(v)[v/x_{\bullet}] \mid v \in V_H \} \vdash S[\ext_H]$ is derivable in $\Logic$.
	\end{enumerate}
\end{definition}

\begin{example}\label{example:hypergraph-L-grammar}
	Let $\Gram$ be a hypergraph $\MILLFO$ grammar with $T \eqdef \{a,b\}$ ($\type(a)=\{\lt,\rt\}$, $\type(b)=\{1,2,3\}$), $X \eqdef \{\lt,\rt\}$, $S \eqdef p(\lt,\rt) \mconj r \mconj r \mconj r$, and with $\triangleright$ consisting of the pairs 
	\begin{itemize}
		\item $a \triangleright q(\lt,\rt)$; \qquad $a \triangleright \forall x. q(x,\lt) \limpl p(x,\rt)$;
		\item $b \triangleright q(2,3) \limpl p(1,1) \limpl p(2,3)$;
		\item $\bullet \triangleright r$; \qquad $\bullet \triangleright \forall y. p(x_\bullet,y)$.
	\end{itemize}
	Consider the hypergraph 
	$H=
	\vcenter{\hbox{{\tikz[baseline=.1ex]{
					\def\X{1.6}
					\def\Y{0.6}
					\def\SHIFT{0}
					\node[node] (2N) at (\SHIFT,\Y) {};
					\node[node, label=left:{\tiny $(\lt)$}] at (\SHIFT,0) (2N1) {};
					\node[node, label=right:{\tiny $(\rt)$}] (2N2) at ($(\SHIFT+\X,0)$) {};
					\node[node] (2N3) at ($(\SHIFT+\X,\Y)$) {};
					\node[hyperedge] (2F) at ($(\SHIFT+\X/2,\Y)$)  {$b$};
					\draw[->, thick] (2N1) to[bend right=20] node[above] {$a$} (2N2);
					\draw[-] (2N) -- node[above] {\tiny 1} (2F);
					\draw[-] (2N1) -- node[above] {\tiny 2} (2F);
					\draw[-] (2N2) -- node[above] {\tiny 3} (2F);
	}}}}
	$ with $V_H=\{v_1,v_2,v_3,v_4\}$ (nodes are enumerated left to right, top to bottom). It belongs to $L(\Gram)$, because the sequent 
	$$
	q(v_3,v_4), q(v_3,v_4) \limpl p(v_1,v_1) \limpl p(v_3,v_4), \forall y. p(v_1,y), r, r, r \vdash p(v_3,v_4) \mconj r \mconj r \mconj r
	$$
	is derivable in $\MILLFO$. In this sequent, the first formula corresponds to the $a$-labeled hyperedge, the second one corresponds to the $b$-labeled hyperedge, the third formula corresponds to $v_1$ and the remaining formulae correspond to $v_2,v_3,v_4$.
	
	The string graph 
	$\SG(aa) = \vcenter{\hbox{{\tikz[baseline=.1ex]{
					\foreach \i in {1,...,3}
					{
						\node[node, label=\ifnumequal{\i}{1}{left}{right}:{\tiny \ifnumequal{\i}{1}{$(\lt)$}{\ifnumequal{\i}{3}{$(\rt)$}{}}}] (V\i) at ($(0.8*\i-0.8,0)$) {};
					}
					\draw[-latex, thick] (V1) -- node[above] {$a$} (V2);
					\draw[-latex, thick] (V2) -- node[above] {$a$} (V3);
	}}}}$ also belongs to $L(\Gram)$, because the sequent
	$$
	q(v_0,v_1), \forall x. q(x,v_1) \limpl p(x,v_2), r, r, r \vdash p(v_0,v_2) \mconj r \mconj r \mconj r
	$$
	is derivable in $\MILLFO$ (the nodes of $\SG(aa)$ from left to right are $v_0,v_1,v_2$).
	
	Finally, the hypergraph 
	$\vcenter{\hbox{{\tikz[baseline=.1ex]{
					\foreach \i in {1,...,4}
					{
						\node[node, label=\ifnumequal{\i}{1}{left}{right}:{\tiny \ifnumequal{\i}{1}{$(\lt)$}{\ifnumequal{\i}{4}{$(\rt)$}{}}}] (V\i) at ($(0.4*\i-0.4,0)$) {};
					}}}}}
	$
	without hyperedges and with nodes $w_1,w_2,w_3,w_4$ is also accepted by $\Gram$, because the following sequent is derivable in $\MILLFO$:
	$$
	\forall y. p(w_1,y), r, r, r \vdash p(w_1,w_4) \mconj r \mconj r \mconj r.
	$$
\end{example}

Let us comment on Definitions \ref{definition:hypergraph-L-grammar} and \ref{definition:language-hypergraph-L-grammar}. The relation $\triangleright$ assigns formulae of $\Logic$ to hyperedge labels from $T$. Besides, it assigns formulae with the free variable $x_\bullet$ (or without free variables) to the distinguished ``node symbol'' $\bullet$. Since, in the theory of hyperedge replacement, it is traditional to consider hypergraphs where only hyperedges are labeled, one might ask why we assign formulas not only to hyperedges but to nodes as well. The answer is that we need to have some control over nodes. If we remove all the parts concerning nodes from Definitions \ref{definition:hypergraph-L-grammar} and \ref{definition:language-hypergraph-L-grammar}, then hypergraph $\Logic$ grammars would completely ignore isolated nodes; this is a minor yet annoying issue. Moreover, each hypergraph language generated by a hypergraph $\Logic$ grammar (for, say, $\Logic=\MILLFO$) would be closed under node identification.
\begin{example}
	Assume that $\SG(aa) \in L(\Gram)$ for a hypergraph $\MILLFO$ grammar $\Gram = \langle T, S, \{\lt,\rt\},\triangleright \rangle$ where nodes do not participate in the grammar formalism. This would mean that there are formulae $A_1,A_2$ with free variables in $\{\lt,\rt\}$ such that $a \triangleright A_1$, $a \triangleright A_2$ and such that the sequent $A_1[v_0/\lt,v_1/\rt], A_2[v_1/\lt,v_2/\rt] \vdash S[v_0/\lt,v_2/\rt]$ is derivable in $\MILLFO$. However, this necessarily implies that the sequent $A_1[v_0/\lt,v_0/\rt], A_2[v_0/\lt,v_1/\rt] \vdash S[v_0/\lt,v_1/\rt]$ is derivable in $\MILLFO$ too, hence the hypergraph $\vcenter{\hbox{{\tikz[baseline=.1ex]{
					\foreach \i in {1,2}
					{
						\node[node, label=\ifnumequal{\i}{1}{left}{right}:{\tiny \ifnumequal{\i}{1}{$(\lt)$}{\ifnumequal{\i}{2}{$(\rt)$}{}}}] (V\i) at ($(0.8*\i-0.8,0)$) {};
					}
					\draw[latex-, thick] (V1) to[out=60,in=120,looseness=30] node[above] {$a$} (V1);
					\draw[-latex, thick] (V1) -- node[above] {$a$} (V2);
	}}}}$ is also accepted by $\Gram$. Consequently, hypergraph $\MILLFO$ grammars without node-formula assignment are not able to generate a language consisting only of string graphs, which is quite undesirable.
\end{example}

Another remark concerning Definition \ref{definition:hypergraph-L-grammar} is why we need the set $X$. This set makes the language generated by a grammar consistent in terms of external nodes: if $H_1,H_2 \in L(\Gram)$, then $\type(H_1)=\type(H_2)=X$. Note that ht-systems and hyperedge replacement grammars \cite{DrewesKH97} are consistent in this sense.

Given a hypergraph $\Logic$ grammar $\Gram$, one can consider only string graphs generated by it and thus associate a string language with $\Gamma$.
\begin{definition}
	The string language $L^\str(\Gram)$ generated by a hypergraph $\Logic$ grammar $\Gram$ is the set $\{w \mid \SG(w) \in L(\Gram)\}$.
\end{definition}
One expects that, normally, string languages generated by hypergraph $\Logic$ grammars should be the same as languages generated by string $\Logic$ grammars. For example, for $\Logic=\MILLFO$, the following proposition holds.
\begin{proposition}\label{proposition:string-hypergraph-MILL1}
	If a language $L$ is generated by a string $\MILLFO$ grammar, then there is a hypergraph $\MILLFO$ grammar $\Gram$ such that $L^{\str}(\Gram)=L$. Conversely, if $\Gram$ is a hypergraph $\MILLFO$ grammar, then $L^{\str}(\Gram) \setminus \{\varepsilon\}$ is generated by some string $\MILLFO$ grammar.
\end{proposition}
This proposition is almost trivial; the only minor technicality is that hypergraph $\MILLFO$ grammars assign formulae to nodes while string $\MILLFO$ grammars are unable to do so. This technicality is the cause why we need to exclude the empty word in the second part of the proposition. See the proof of this proposition in Appendix \ref{appendix:proof-proposition:string-hypergraph-MILL1}.

\subsection{Hypergraph Transformation Rules as MILL1 Formulae}

Our main goal now is to describe a relation between hypergraph first-order linear logic grammars and hypergraph transformation systems. We start with showing how a ht-rule is encoded by a $\MILLFO$ formula. Let us fix a unary predicate $\nu(x)$; informally, it is understood as ``$x$ is a node''. Given a hypergraph $H$, let us treat its hyperedge labels as predicate symbols. Namely, if $\lab_H(e)=a$, then let us assume that the elements of $\type(a)$ are enumerated, i.e. $\type(a) = \{\sigma_1,\ldots,\sigma_n\}$; for any function $h:\type(A) \to \Var$, let $\lab_H(e)[h]$ stand for the formula $a(h(\sigma_1),\ldots,h(\sigma_n))$. The arity of $a$ is $n=\vert\type(a)\vert$. 
\begin{definition}\label{definition:diagram}
	The \emph{diagram of a hypergraph $H$} is the multiset $$\diag(H) \eqdef \{\lab_H(e)[\att_H(e)] \mid e \in E_H \} \cup \{ \nu(v) \mid v \in V_H \}.$$
	Let $D(H) \eqdef \exists \vec{v} \bigotimes \mathcal{D}(H)$ where $\vec{v}$ is the list of nodes in $V_H \setminus \ran(\ext_H)$.
\end{definition}

\begin{example}
	Let $H = 
	\vcenter{\hbox{{\tikz[baseline=.1ex]{
					\def\X{1.6}
					\def\Y{0.6}
					\def\H{1.3}
					\def\SHIFT{0}
					\node[node] (2N) at (\SHIFT,\Y) {};
					\node[node] at (\SHIFT,0) (2N1) {};
					\node[node] (2N2) at ($(\SHIFT+\X,0)$) {};
					\node[node] (2N3) at ($(\SHIFT+\X,\Y)$) {};
					\node[hyperedge] (2F) at ($(\SHIFT+\X/2,\Y)$)  {$C$};
					\node[hyperedge] (2F2) at ($(\SHIFT+\X+0.7,0)$)  {$Y$};
					\draw[->, thick] (2N1) to[bend right=20] node[above] {$D$} (2N2);
					\draw[<-, thick] (2N1) to[bend right=60] node[below] {$X$} (2N2);
					\draw[-] (2N2) -- (2F2);
					\draw[-] (2N) -- node[above] {\tiny 1} (2F);
					\draw[-] (2N1) -- node[above] {\tiny 2} (2F);
					\draw[-] (2N2) -- node[above] {\tiny 3} (2F);
	}}}}$ be a hypergraph such that $V_H = \{v_1,v_2,v_3,v_4\}$. Then
	$
		\diag(H) = C(v_1,v_3,v_4) , D(v_3,v_4) , X(v_4,v_3), Y(v_4), \nu(v_1), \nu(v_2), \nu(v_3), \nu(v_4).
	$
\end{example}

\begin{definition}
	Given a ht-rule $p = (H \to H^\prime)$, let \\
	$\fm(p) \eqdef \forall \vec{u} \left(D(H^\prime) \limpl D(H)[\chi_p] \right)$ where 
	\begin{itemize}
		\item $\vec{u}$ is the list of nodes in $\ran(\ext_{H^\prime})$;
		\item $\chi_p$ is a substitution function defined on $\ran(\ext_H)$ as follows: $\chi_p(\ext_H(\sigma)) = \ext_{H^\prime}(\sigma)$.
	\end{itemize}
\end{definition}

The formula $\fm(p)$ is closed. Note that $\chi_p$ is well defined because $\ext_H$ is injective.

\begin{example}
	Consider the ht-rule 
	$$p = \vcenter
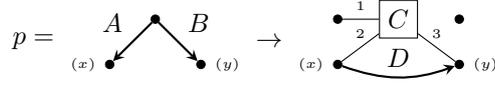
{\hbox{{\tikz[baseline=.1ex]{
					\node[node] (N) at (0,0.6) {};
					\node[node, label=left:{\tiny $(x)$}] (N1) at (-0.6,0) {};
					\node[node, label=right:{\tiny $(y)$}] (N2) at (0.6,0) {};
					\draw[->, thick] (N) -- node[above left] {$A$} (N1);
					\draw[->, thick] (N) -- node[above right] {$B$} (N2);
					\def\X{1.6}
					\def\Y{0.6}
					\def\SHIFT{2.4}
					\node at ($({(\SHIFT+0.6)/2},\Y/2)$) {$\to$};
					\node[node] (2N) at (\SHIFT,\Y) {};
					\node[node, label=left:{\tiny $(x)$}] at (\SHIFT,0) (2N1) {};
					\node[node, label=right:{\tiny $(y)$}] (2N2) at ($(\SHIFT+\X,0)$) {};
					\node[node] (2N3) at ($(\SHIFT+\X,\Y)$) {};
					\node[hyperedge] (2F) at ($(\SHIFT+\X/2,\Y)$)  {$C$};
					\draw[->, thick] (2N1) to[bend right=20] node[above] {$D$} (2N2);
					\draw[-] (2N) -- node[above] {\tiny 1} (2F);
					\draw[-] (2N1) -- node[above] {\tiny 2} (2F);
					\draw[-] (2N2) -- node[above] {\tiny 3} (2F);
					}}}}
				$$ 
	Let $u_1,u_2,u_3$ be the nodes of the left-hand side hypergraph (from left to right) and let $v_1,v_2,v_3,v_4$ be the nodes of the right-hand side hypergraph. Then 
	\begin{multline*}
		\fm(p) = \forall v_3.\forall v_4. 
		\big( 
		(\exists v_1.\exists v_2.C(v_1,v_3,v_4)\mconj D(v_3,v_4) \mconj \nu(v_1) \mconj \nu(v_2) \mconj \nu(v_3) \mconj \nu(v_4) ) 
		\limpl
		\\(\exists u_2. A(u_2,v_3) \mconj B(u_2,v_4) \mconj \nu(v_3) \mconj \nu(u_2) \mconj \nu(v_4))
		\big).
	\end{multline*}
	
\end{example}

The main lemma about $\fm(p)$ is presented below.
\begin{lemma}\label{lemma:main}
	Let $P,P^\prime$ be multisets of ht-rules and let $G,G^\prime$ be two hypergraphs with injective $\ext_G,\ext_{G^\prime}$. The sequent 
	\begin{equation}\label{equation:sequent-fm}
		\{\bang \fm(r) \mid r \in P \}, \{\fm(r) \mid r \in P^\prime \}, \diag(G^\prime) \vdash D(G)[\chi_{G \to G^\prime}]
	\end{equation}
	is derivable in $\ILLFO$ if and only if there exists a derivation of a hypergraph isomorphic to $G^\prime$ from $G$ which uses each rule from $P^\prime$ exactly once and that can use rules from $P$ any number of times.
\end{lemma}

The proof of this lemma, placed in Appendix \ref{appendix:proof-lemma:main}, is quite technical; its idea is to do a proof search using focusing \cite{Andreoli92}. Using Lemma \ref{lemma:main} we can prove the following theorem.

\begin{theorem}\label{theorem:ILL1G=RE}
	Hypergraph $\ILLFO$ grammars generate the class of all recursively enumerable hypergraph languages.
\end{theorem}
\begin{proof}
	Clearly, all languages generated by hypergraph $\ILLFO$ grammars are recursively enumerable. To prove the converse, according to Proposition \ref{proposition:hts=re}, it suffices to show that any ht-system $\Gram = \langle N, T, P, S\rangle$ can be converted into a hypergraph $\ILLFO$ grammar generating the same language. 
	Define the hypergraph $\ILLFO$ grammar $\Gram^\prime = \langle T, S^\prime, \type(S), \triangleright \rangle$ such that
	\begin{enumerate}
		\item $S^\prime \eqdef \bigotimes \{\bang \fm(r) \mid r \in P\} \limpl D(S)[h_S]$ where $h_S(\ext_S(\sigma))=\sigma$ for $\sigma \in \type(S)$;
		\item $a \triangleright a(\sigma_1,\ldots,\sigma_n)$ for $a \in T$ and $\type(a)=\{\sigma_1,\ldots,\sigma_n\}$;
		\item $\bullet \triangleright \nu(x_\bullet)$.
	\end{enumerate}
	A hypergraph $H$ is accepted by $\Gram^\prime$ if and only if the following sequent is derivable in $\ILLFO$:
	$$
	\{\lab_H(e)[\att_H(e)] \mid e \in E_H\}, \{\nu(v) \mid v \in V_H\} \vdash S^\prime[\ext_H].
	$$
	Note that the antecedent of this sequent is the diagram of $H$ and that the succedent equals $\bigotimes \{\bang \fm(r) \mid r \in P\} \limpl D(S)[h_S][\ext_H] = \bigotimes \{\bang \fm(r) \mid r \in P\} \limpl D(S)[\chi_{S \to H}]$. By invertibility of the rules $(\limpl R)$ and $(\mconj L)$, the above sequent is equiderivable with the one
	$$
	\{\bang \fm(r) \mid r \in P\}, \diag(H) \vdash D(S)[\chi_{S \to H}]
	$$
	By Lemma \ref{lemma:main}, derivability of the latter sequent is equivalent to the fact that $H$ is derivable from $S$ using rules from $P$, i.e. that $H \in L(\Gram)$. Thus, $L(\Gram)=L(\Gram^\prime)$.
\end{proof}

This result justifies soundness of Definition \ref{definition:hypergraph-L-grammar}: if hypergraph grammars based on $\ILLFO$, which includes the powerful exponential modality, were not Turing-complete, this would indicate that we do not have enough control in the grammar formalism. For example, if we did not include node-formula assignment in Definition \ref{definition:hypergraph-L-grammar}, then Theorem \ref{theorem:ILL1G=RE} would be false.

\subsection{Hypergraph MILL1 Grammars}\label{subsection:hypergraph-MILL1-grammars}

We proceed to investigating expressive power of hypergraph $\MILLFO$ grammars. They are clearly less expressive than ht-systems because they generate only languages from NP. Nevertheless, it turns out that they as powerful as ht-systems in which the length of a derivation of a hypergraph is bounded by a linear function w.r.t. the size of the hypergraph.
\begin{definition}
	The size of a hypergraph $H$ denoted by $\vert H \vert$ is the number of nodes and hyperedges in $H$. A \emph{linear-time hypergraph transformation system} is a ht-system $\Gram = \langle N, T, P, S\rangle$ for which there is $c \in \Nat$ (a \emph{time constant}) such that, for each $H \in L(\Gram)$, there is a derivation $S \Rightarrow^\ast H$ with at most $c \cdot \vert H\vert$ steps.
\end{definition}

Adding linear-time bound to formal grammars has a long history. Linear-time type-0 Chomsky grammars were studied in \cite{Book71,Gladkii64}; linear-time one-tape Turing machines were studied in \cite{Tadaki10}; linear-time branching vector addition systems were studied in \cite{Pshenitsyn23} in the context of commutative Lambek grammars. However, to our best knowledge, linear-time graph grammars have not appeared in the literature. Linear-time ht-systems may be considered as the hypergraph counterpart of linear-time type-0 grammars studied in \cite{Book71,Gladkii64}, so it is quite a natural formalism. The main result concerning them is presented below.
\begin{theorem}\label{theorem:MILL1G>LTHTS}
	Each linear-time ht-system can be converted into an equivalent hypergraph $\MILLFO$ grammar.
\end{theorem}
\begin{proof}[Proof of Theorem \ref{theorem:MILL1G>LTHTS}]
	Let $\Gram = \langle N, T, P, S\rangle$ be a linear-time ht-system with the time constant $c \in \Nat$. Define the hypergraph $\MILLFO$ grammar $\Gram^\prime = \langle T, S^\prime, \type(S), \triangleright \rangle$ such that
	\begin{enumerate}
		\item $S^\prime \eqdef D(S)[h_S]$ where $h_S(\ext_S(\sigma))=\sigma$ for $\sigma \in \type(S)$;
		\item $a \triangleright a(\sigma_1,\ldots,\sigma_n) \mconj \fm(p_1) \mconj \ldots \mconj \fm(p_k)$ for $a \in T$, $\type(a)=\{\sigma_1,\ldots,\sigma_n\}$, $0 \le k \le c$, and for $p_1,\ldots,p_k \in P$;
		\item $\bullet \triangleright \nu(x_\bullet) \mconj \fm(p_1) \mconj \ldots \mconj \fm(p_k)$ for $0 \le k \le c$, and for $p_1,\ldots,p_k \in P$.
	\end{enumerate}
	A hypergraph $H$ is accepted by $\Gram^\prime$ if and only if, for each $e \in E_H$ (and for each $v \in V_H$), there exist at most $c$ rules from $P$, which we denote by $p_1^e,\ldots,p_{k(e)}^e$ (by $q_1^v,\ldots,q_{k(v)}^v$ resp.), such that the sequent
	\begin{multline*}
		\{\lab_H(e)[\att_H(e)]\mconj \fm(p^e_1) \mconj \ldots \mconj \fm(p^e_{k(e)})  \mid e \in E_H\}, 
		\\
		\{\nu(v)\mconj \fm(q^v_1) \mconj \ldots \mconj \fm(q^v_{k(v)}) \mid v \in V_H\} \vdash S^\prime[\ext_H]
	\end{multline*}
	is derivable in $\MILLFO$. By invertibility of $(\mconj L)$, this is equivalent to the fact that there exists a multiset $P^\prime$ of cardinality at most $c \cdot \vert E_H \vert + c \cdot \vert V_H \vert = c \cdot \vert H \vert$ such that all its elements are from $P$ and such that the following sequent is derivable in $\MILLFO$:
	$$
		\{\fm(p) \mid p \in P^\prime\}, \diag(H) \vdash D(S)[\chi_{S \to H}]
	$$
	By Lemma \ref{lemma:main}, this is equivalent to the fact that there is a derivation of a hypergraph $H^\prime$ isomorphic to $H$ from $S$ that uses each rule from $P^\prime$ exactly once. Existence of $P^\prime$ satisfying these properties is equivalent to the fact that there is a derivation of $H^\prime$ from $S$ of length at most $c \cdot \vert H \vert$, which is equivalent to $H \in L(\Gram)$. Thus, $L(\Gram)=L(\Gram^\prime)$.
\end{proof}
The question arises whether the converse holds as well. We are not going to address it in the article because of space-time limitations and also because we are mainly interested in lower bounds for the class of hypergraph $\MILLFO$ grammars. Still, we claim that it is possible to convert each hypergraph $\MILLFO$ grammar into a linear-time ht-system \emph{with non-injective rules}, i.e. with rules $H \to H^\prime$ where $\ext_{H^\prime}$ is allowed to be non-injective. This could be done by a straightforward (yet full of tiring technical details) encoding of $\MILLFO$ inference rules by hypergraph transformations. We leave proving that for the future work.

Speaking of upper bounds, we noted that all languages generated by hypergraph $\MILLFO$ grammars are in NP; besides, as we shall show in Theorem \ref{theorem:MILL1G-NPC}, there is an NP-complete language of string graphs generated by a hypergraph $\MILLFO$ grammars, so the NP upper bound is accurate. 

Let us further explore properties of the class of languages generated by hypergraph $\MILLFO$ grammars. It turns out to be closed under intersection.
\begin{theorem}\label{theorem:hypergraph-MILL1-grammars-intersection}
	Languages generated by hypergraph $\MILLFO$ grammars are closed under intersection.
\end{theorem}
\begin{proof}
	Let $\Gram_i = \langle T_i, S_i, X_i, \triangleright_i \rangle$ be two hypergraph $\MILLFO$ grammars; we aim to construct a hypergraph $\MILLFO$ grammar generating $L(\Gram_1) \cap L(\Gram_2)$. Let us assume without loss of generality that $T_1=T_2 = T$; also, let us assume that predicate symbols used by $\Gram_1$ and $\Gram_2$ are disjoint. Let us denote the set of subformulas of formulas occuring in $\Gram_i$ by $\Fm_i$ ($i=1,2$). Note that, if $X_1 \ne X_2$, then $L(\Gram_1) \cap L(\Gram_2) = \emptyset$ (hypergraphs in these languages would have different types); thus, we can assume that $X_1=X_2$ (let us denote this set by $X$).
	
	Define the grammar $\Gram \eqdef \langle T, S, X, \triangleright \rangle$ where $S = S_1 \mconj S_2$ and $\triangleright$ is the smallest relation such that $a \triangleright A_1 \mconj A_2$ holds whenever $a \in T \cup \{\bullet\}$ and $a \triangleright_i A_i$ for $i=1,2$.
	\begin{lemma}[splitting lemma]\label{lemma:splitting}
	\leavevmode
	\begin{enumerate}
		\item If the sequent $A_1,\ldots,A_n \vdash B$ is derivable in $\MILLFO$ such that $A_i \in \Fm_1 \cup \Fm_2$ and $B \in \Fm_k$ for some $k \in \{1,2\}$, then all $A_i$ are also from $\Fm_k$.
		\item The sequent $A_1,\ldots,A_n,B_1,\ldots,B_m \vdash A \mconj B$ such that $A_i$ and $A$ are from $\Fm_1$ and $B_i$, $B$ are from $\Fm_2$ is derivable in $\MILLFO$ if and only if the sequents $A_1,\ldots,A_n \vdash A$ and $B_1,\ldots,B_m \vdash B$ are derivable in $\MILLFO$.
	\end{enumerate}
	\end{lemma}
	Both statements are proved jointly by straightforward induction on the length of a derivation.
	
	A hypergraph $H$ is accepted by $\Gram$ if and only if, for $i=1,2$, there exist functions $h^i_V:V_H \to \Fm_i$ and $h^i_E:E_H \to \Fm_i$ such that $\bullet \triangleright_i h^i_V(v)$ for $v \in V_H$, $\lab_H(e) \triangleright_i h^i_E(v)$ for $e \in E_H$, and the following sequent is derivable in $\MILLFO$:
	$$
	\{(h_E^1(e)\mconj h_E^2(e))[\att_H(e)] \mid e \in E_H\},
	\{(h_V^1(v)\mconj h_V^2(v))[v/x_\bullet] \mid v \in V_H\}
	\vdash
	(S_1 \mconj S_2)[\ext_H].
	$$
	By invertibility of $(\mconj L)$ and splitting lemma, this is equivalent to derivability of the sequents $\{h_E^i(e)[\att_H(e)] \mid e \in E_H\},
	\{h_V^i(v)[v/x_\bullet] \mid v \in V_H\}
	\vdash
	S_i[\ext_H]$ for $i=1,2$, hence is equivalent to the fact that $H$ belongs to $L(\Gram_1) \cap L(\Gram_2)$.
\end{proof}
An analogous technique cannot be used for Lambek categorial grammars because the Lambek calculus is non-commutative, and splitting lemma does not hold for it. In fact, since Lambek categorial grammars generate context-free languages \cite{Pentus93}, which are not closed under intersection, a similar result does not hold for Lambek categorial grammars.

\section{Expressive Power of String MILL1 Grammars}\label{section:MILL1-grammars}

Now, let us apply the results and techniques developed for hypergraph $\MILLFO$ grammars to describing the class of languages generated by string $\MILLFO$ grammars. First, Proposition \ref{proposition:string-hypergraph-MILL1} and Theorem \ref{theorem:MILL1G>LTHTS} imply the following lower bound.
\begin{corollary}\label{corollary:string-MILL1-grammars}
	If $\Gram = \langle N, T, P, S \rangle$ is a linear-time ht-system, then $\{w \in T^+ \mid \SG(w) \in L(\Gram)\}$ is generated by a string $\MILLFO$ grammar.
\end{corollary}

Next, the following result can be proved in the same way as Theorem \ref{theorem:hypergraph-MILL1-grammars-intersection}.
\begin{theorem}
	Languages generated by string MILL1 grammars are closed under intersection.
\end{theorem}
(Note that we cannot directly infer this theorem from Proposition \ref{proposition:string-hypergraph-MILL1} and Theorem \ref{theorem:hypergraph-MILL1-grammars-intersection} because of the empty string issue.)
Since languages generated by string MILL1 grammars contain all context-free languages, they also contain their finite intersections, in particular, the language 
\begin{multline*}
	\{(a^nb^n)^n \mid n > 0 \} = \{a^{n_1}b^{n_1}\ldots a^{n_k}b^{n_k} \mid n_1,\ldots,n_k \in \Nat\} \\ \cap \{a^{k}b^{n_1}a^{n_1}\ldots b^{n_{k-1}} a^{n_{k-1}} b^l \mid k,l,n_1,\ldots,n_{k-1} >0 \}.
\end{multline*}
It is simple to prove that languages generated by string $\MILLFO$ grammars are also closed under letter-to-letter homomorphisms, so the language $\{a^{2n^2} \mid n \in \Nat\}$ can be generated by a string $\MILLFO$ grammar as well. Thus, string $\MILLFO$ grammars generate languages with non-semilinear Parikh images.

Note that, for Lambek categorial grammars, there is imbalance between expressive power and algorithmic complexity. Namely, on the one hand, Lambek categorial grammars generate exactly context-free languages, all of which are polynomially parsable, but on the other hand, parsing in the Lambek calculus is an NP-complete problem \cite{Pentus06}. This is not the case for $\MILLFO$ grammars, as the following theorem shows.
\begin{theorem}\label{theorem:MILL1G-NPC}
	String $\MILLFO$ grammars generate an NP-complete language.
\end{theorem}
A straightforward proof of this theorem can be found in Appendix \ref{appendix:proof-theorem:MILL1G-NPC}. There, we present a linear-time ht-system generating an NP-complete language of string graphs. Here, we shall present another proof that relies on a result from \cite{Book78}.
\begin{proof}
	Consider a string rewriting system $\mathcal{S} = \langle N,T,P,S \rangle$ where $P$ is a finite set of rules of the form $\alpha \to \beta$ where $\alpha,\beta \in (N\cup T)^\ast$ are arbitrary strings and $S \in (N \cup T)^\ast$ is the start string. Recall that $\eta \alpha \theta \Rightarrow_{\mathcal{S}} \eta \beta \theta$ whenever $(\alpha \to \beta) \in P$ (and no other pairs of strings are in the relation $\Rightarrow_{\mathcal{S}}$); recall also that $L(\mathcal{S}) = \{w \in T^\ast \mid S \Rightarrow_{\mathcal{S}}^\ast w \}$. 
	
	Let us convert $\mathcal{S}$ into the ht-system $\mathcal{G} = \langle N, T, P^\prime, \SG(S)\rangle$ where the symbols from $N \cup T$ have the type $\{\lt,\rt\}$ and $P^\prime = \{\SG(\alpha) \to \SG(\beta) \mid (\alpha \to \beta) \in P\}$. (Let us assume that, in each rule $\alpha \to \beta$, $\alpha$ and $\beta$ are nonempty.) It is straightforward to check that $L(\mathcal{G}) = \{\SG(w) \mid w \in L(\mathcal{S})\}$. Finally, \cite[Theorem 1]{Book78} says that there is a linear-time string rewriting system $\mathcal{S}$ that generates an NP-complete language. Thus, if one constructs $\mathcal{G}$ from it and then applies Corollary \ref{corollary:string-MILL1-grammars}, then they obtain a string $\MILLFO$ grammar generating an NP-complete language.
\end{proof}

One of the reviewers pointed out that string $\MILLFO$ grammars generating nonsemilinear languages are too much for linguistic applications, where it is widely assumed that natural languages are semilinear. The reviewer asked whether restricting $\MILLFO$ to the fragment where each (free or bound) variable occurs in each formula at most twice results in languages generated by the corresponding class of grammars being semilinear. I am afraid that, even with this restriction, string $\MILLFO$ grammars generate some non-semilinear and NP-complete languages. Let us consider the above construction from Theorem \ref{theorem:MILL1G-NPC}. E.g.~if $AB \to BCD$ is a string rewriting rule, then 
\begin{multline*}
	\fm(\SG(AB) \to \SG(BCD)) = \forall v_1.\forall v_2. \big[\left(\exists x.A(v_1,x)\otimes B(x,v_2) \otimes \nu(v_1) \otimes \nu(x) \otimes \nu(v_2)\right) \limpl  \\ \left(\exists y. \exists z.B(v_1,y)\otimes C(y,z) \otimes D(z,v_2) \otimes \nu(v_1) \otimes \nu(y) \otimes \nu(z) \otimes \nu(v_2)\right)\big].
\end{multline*}
In a formula of the form $\fm(\SG(\alpha) \to \SG(\beta))$, each variable occurs at most four times. However, we claim that one could remove $\nu$ predicates everywhere, which would result in each variable occurring in each formula at most twice. We also claim that Lemma \ref{lemma:main} would remain true without $\nu$ predicates in formulas in case where we consider only string graphs. Thus, we can use Book's result from \cite{Book78} to generate an NP-complete and a non-semilinear language by a grammar over the restricted fragment of $\MILLFO$. We do not provide formal proofs but leave them for the future work.

\section{Hypergraph Language Semantics}\label{section:semantics}

Now, let us proceed to model-theoretic investigations into $\MILLFO$. Our objective is to generalise language models for the Lambek calculus to $\MILLFO$ and to devise \emph{hypergraph language models}. This will enable one to regard $\MILLFO$ as a logic for reasoning about hypergraph resources. The most important question is how the composition of such resources should be defined: if $H_1,H_2$ are two hypergraphs, then how should one understand ``$H_1 \mconj H_2$''? 

Language semantics for the Lambek calculus, algebraically speaking, is a mapping of $\mathrm{L}$ formulae to a free semigroup of words which interprets product $A \cdot B$ as elementwise product of interpretations of $A$ and $B$. What is the hypergraph counterpart of free semigroups? In the field of hyperedge replacement, there are algebras of hypergraphs called HR-algebras \cite{Courcelle90,CourcelleE12,Rozenberg97}. They include the parallel composition operation and source manipulating operations. Parallel composition is a way of gluing hypergraphs defined as follows.
\begin{definition}\label{definition:pc}
	Let $H_1$ and $H_2$ be hypergraphs. Let $\sim$ be the smallest equivalence relation on $V_{H_1} \sqcup V_{H_2}$ such that $\ext_{H_1}(\sigma) \sim \ext_{H_2}(\sigma)$ for $\sigma \in \type(H_1) \cap \type(H_2)$. Then, \emph{parallel composition $H_1 \pc H_2$} is the hypergraph $H$ such that $\type(H) = \type(H_1) \cup \type(H_2)$; $V_H = (V_{H_1} \sqcup V_{H_2})/\sim$, $E_H = E_{H_1} \sqcup E_{H_2}$; $\att_H(e)(\sigma) = [\att_{H_i}(e)(\sigma)]_\sim$, $\lab_H(e)=\lab_{H_i}(e)$ for $e \in E_{H_i}$; $\ext_H(\sigma)=[\ext_{H_i}(\sigma)]_\sim$ for $\sigma \in \type(H_1) \cup \type(H_2)$.
\end{definition}
Informally, $H_1 \pc H_2$ is obtained by taking the disjoint union of $H_1$ and $H_2$ and fusing $\ext_{H_1}(\sigma)$ with $\ext_{H_2}(\sigma)$ for $\sigma \in \type(H_1) \cap \type(H_2)$. This operation is illustrated on Figure \ref{figure:pc-sub}. Note that parallel composition is associative and commutative.

\begin{figure}
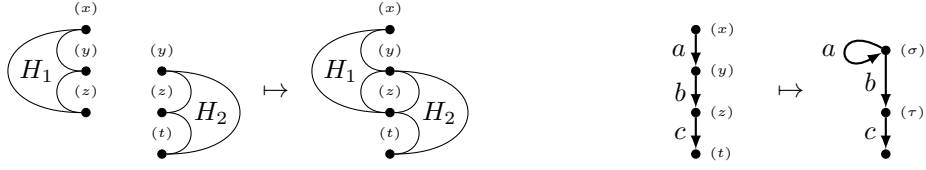

	\centering
	
	{\tikz[baseline=.1ex]{
			\def\X{1}
			\def\XX{4}
			\def\Y{0.55}
			\def\BEND{30}
			\node[node, label=above:{\tiny $(x)$}] (V1) at ($(0,2*\Y)$) {};
			\node[node, label=above:{\tiny $(y)$}] (V2) at ($(0,1*\Y)$) {};
			\node[node, label=above:{\tiny $(z)$}] (V3) at ($(0,0*\Y)$) {};
			\node at ($(-0.65,1*\Y)$) {$H_1$};
			\draw[-] (V1) to[bend right=90,looseness=3] (V3);
			\draw[-] (V1) to[bend right=90,looseness=2] (V2);
			\draw[-] (V2) to[bend right=90,looseness=2] (V3);
			\node[node, label=above:{\tiny $(y)$}] (2V1) at ($(\X,1*\Y)$) {};
			\node[node, label=above:{\tiny $(z)$}] (2V2) at ($(\X,0*\Y)$) {};
			\node[node, label=above:{\tiny $(t)$}] (2V3) at ($(\X,{(-1)*\Y})$) {};
			\node at ($(\X+0.65,0*\Y)$) {$H_2$};
			\draw[-] (2V1) to[bend left=90,looseness=3] (2V3);
			\draw[-] (2V1) to[bend left=90,looseness=2] (2V2);
			\draw[-] (2V2) to[bend left=90,looseness=2] (2V3);
			\node at ($({(\XX+\X)/2},{0.5*\Y})$) {$\mapsto$};
			\node[node, label=above:{\tiny $(x)$}] (3V1) at ($(\XX,2*\Y)$) {};
			\node[node, label=above:{\tiny $(y)$}] (3V2) at ($(\XX,1*\Y)$) {};
			\node[node, label=above:{\tiny $(z)$}] (3V3) at ($(\XX,0*\Y)$) {};
			\node[node, label=above:{\tiny $(t)$}] (3V4) at ($(\XX,-\Y)$) {};
			\node at ($(\XX-0.65,1*\Y)$) {$H_1$};
			\draw[-] (3V1) to[bend right=90,looseness=3] (3V3);
			\draw[-] (3V1) to[bend right=90,looseness=2] (3V2);
			\draw[-] (3V2) to[bend right=90,looseness=2] (3V3);
			\node at ($(\XX+0.65,0*\Y)$) {$H_2$};
			\draw[-] (3V2) to[bend left=90,looseness=3] (3V4);
			\draw[-] (3V2) to[bend left=90,looseness=2] (3V3);
			\draw[-] (3V3) to[bend left=90,looseness=2] (3V4);
	}}
	\hspace{2cm}
	{\tikz[baseline=.1ex]{
			\def\X{1}
			\def\XX{2.5}
			\def\Y{0.55}
			\def\BEND{30}
			\node[node, label=right:{\tiny $(x)$}] (V1) at ($(0,2*\Y)$) {};
			\node[node, label=right:{\tiny $(y)$}] (V2) at ($(0,1*\Y)$) {};
			\node[node, label=right:{\tiny $(z)$}] (V3) at ($(0,0*\Y)$) {};
			\node[node, label=right:{\tiny $(t)$}] (V4) at ($(0,-\Y)$) {};
			\draw[-latex,thick] (V1) -- node[left] {$a$} (V2);
			\draw[-latex,thick] (V2) -- node[left] {$b$} (V3);
			\draw[-latex,thick] (V3) -- node[left] {$c$} (V4);
			\node at ($({\XX/2},{0.5*\Y})$) {$\mapsto$};
			\node[node, label=right:{\tiny $(\sigma)$}] (3V1) at ($(\XX,1.5*\Y)$) {};
			\node[node, label=right:{\tiny $(\tau)$}] (3V2) at ($(\XX,0*\Y)$) {};
			\node[node] (3V3) at ($(\XX,-\Y)$) {};
			\draw[-latex,thick] (3V1) to[in=210,out=150,looseness=30] node[left] {$a$} (3V1);
			\draw[-latex,thick] (3V1) -- node[left] {$b$} (3V2);
			\draw[-latex,thick] (3V2) -- node[left] {$c$} (3V3);
	}}

	\caption{Scheme of parallel composition of $H_1$ and $H_2$ (left) and an example of substitution of a hypergraph according to $h$ where $h(x)=h(y)=\sigma$, $h(z)=\tau$, and $h(t)$ is undefined (right).}
	\label{figure:pc-sub}
\end{figure}

Other operations used in HR-algebras allow one to reassign selectors to external nodes, to make an external node non-external, or to fuse some external nodes. We shall use an operation that unifies all the three manipulations; we shall call it \emph{substitution}.
\begin{definition}\label{definition:sub}
	Given a hypergraph $H$ and a partial function $h \subseteq \Sigma \times \Sigma$, let $\sim$ be the smallest equivalence relation such that $\ext_H(\sigma_1) \sim \ext_H(\sigma_2)$ whenever $h(\sigma_1) =h(\sigma_2)$. Then $\sub_h(H)$ is the hypergraph $H^\prime$ such that $\type(H^\prime) = h(\type(H))$; $V_{H^\prime} = V_H/\sim$; $E_{H^\prime}=E_H$; $\att_{H^\prime}(e)(\sigma) = [\att_H(e)(\sigma)]_\sim$, $\lab_{H^\prime}(e)=\lab_H(e)$ for $e \in E$; $\ext_{H^\prime}(h(\sigma)) = [\ext_H(\sigma)]_\sim$.
\end{definition}

One can compare the substitution operation with the operations of source renaming and source fusion from \cite{Rozenberg97} and verify that the former one is interdefinable with the latter ones (in presence of parallel composition). Hence, one can use $\pc$ and $\sub_h$ as basic operations of HR-algebras. Nicely, exactly these operations can also be used for defining hypergraph language models, which we are going to do now. Let $K_0 \eqdef \langle \emptyset, \emptyset, \emptyset, \emptyset, \emptyset \rangle$ be the empty hypergraph. From now on, $\Sigma=\Var$ (i.e. selectors and variables are the same objects).
\begin{definition}\label{definition:model}
	A \emph{hypergraph language model} is a pair $\langle T, u \rangle$ where $T$ is a $\Sigma$-typed alphabet and $u:\Fm(\MILLFO) \to \mathcal{P}(\mathcal{H}(T))$ is a function mapping formulas of $\MILLFO$ to sets of abstract hypergraphs over $T$ which satisfies the following conditions:
	\begin{enumerate}
		\item\label{item:def-model-sub} $\sub_h(u(A)) \subseteq u(A[h])$ for any total function $h:\Sigma \to \Sigma$;
		\item $u(A \mconj B) = u(A) \pc u(B)$;
		\item $u(A \limpl B) = \{H \mid \forall H^\prime \in u(A)\, \left(H \pc H^\prime \in u(B) \right) \}$;
		\item $u(\exists x A) = \bigcup\limits_{y \in \Var} u(A[y/x])$;
		\item $u(\forall x A) = \bigcap\limits_{y \in \Var} u(A[y/x])$.
	\end{enumerate}
	A sequent $A_1,\ldots,A_n \yields B$ is true in this model if $u(A_1) \pc \ldots \pc u(A_n) \subseteq u(B)$ (for $n=0$, if $K_0 \in u(B)$).
\end{definition}
This semantics can be viewed as an instance of intuitionistic phase semantics \cite{KanovichOT06} with the commutative monoid $(\HG(T),\pc,K_0)$ and with the trivial closure operator $\mathrm{Cl}(X)=X$.

The first property of $u$ in Definition \ref{definition:model} relates substitution as a logical operation to substitution as a hypergraph transformation; without it, $u(A)$ and $u(A[h])$ would be unrelated, which is undesirable. Besides, this property is used to prove correctness. Note that, if $h:\Sigma \to \Sigma$ is a bijection, then $u(A[h]) = \sub_h(u(A))$ (apply property 1 twice). Quantifiers are interpreted as additive conjunction and disjunction, which reflects their behaviour correctly. 
\begin{lemma}\label{lemma:val-sub-monotone}
	\leavevmode
	\begin{enumerate}
		\item $u(A) \pc u(B) \subseteq u(C)$ if and only if $u(A) \subseteq u(B \limpl C)$.
		\item If $u(A) \subseteq u(B)$, then $u(A[h]) \subseteq u(B[h])$ for any substitution $h$.
	\end{enumerate}
\end{lemma}
\begin{proof}
	\begin{enumerate}
		\item Trivially follows from the definition of a model.
		\item Clearly, $K_0 \pc H = H$ for any $H$. Therefore, $u(A) \subseteq u(B)$ implies $K_0 \in u(A \limpl B)$. Consequently, according to item \ref{item:def-model-sub} of Definition \ref{definition:model}, $K_0 = \sub_h(K_0) \in u(A[h] \limpl B[h])$. This implies that $u(A[h]) \subseteq u(B[h])$, as desired. \qedhere
	\end{enumerate}
\end{proof}

The main result is soundness of $\MILLFO$ and completeness of its fragment $\MILLFO(\limpl,\forall)$ w.r.t. hypergraph language models.
\begin{theorem}\label{theorem:completeness}
	\leavevmode
	\begin{enumerate}
		\item $\MILLFO$ is sound w.r.t. hypergraph language models.
		\item $\MILLFO(\limpl,\forall)$ is complete w.r.t. hypergraph language models.
	\end{enumerate}
\end{theorem}
\begin{proof}
	Soundness can be checked straightforwardly by showing that the conclusion of each rule with true premises is also true. The only nontrivial cases are the rules $(\exists L)$ and $(\forall R)$. Assume that $\Gamma, A[z/x] \vdash B$ where $z$ does not occur freely in $\Gamma,B$ is true in a model $\langle T, u \rangle$. Without loss of generality, let us assume that $\Gamma$ is empty (otherwise, we can move it to the succeedent using Lemma \ref{lemma:val-sub-monotone}). Under this assumption, we are given that $u(A[z/x]) \subseteq u(B)$. By Lemma \ref{lemma:val-sub-monotone}, for each $y \in \Var$, $u(A[y/x]) = u(A[z/x][y/z]) \subseteq u(B[y/z]) = u(B)$. This proves that $u(\exists x A) \subseteq u(B)$. The case $(\forall R)$ is dealt with similarly.
	
	Completeness of $\MILLFO(\limpl,\forall)$ is proved by constructing a canonical model $\langle \mathrm{T}, \mathrm{u} \rangle$. Before we do this, let us introduce a new notion used in the construction. 
	\begin{definition}\label{definition:formula-circ}
		Assume that a countable sequence of variables $\xi_1,\xi_2,\ldots$ is fixed. Given a formula $A$, let us read it from left to right and replace the $i$-th occurrence of a free variable from the left by $\xi_i$. We denote the resulting formula by $A^\circ$. 
	\end{definition}
	For example, $(\exists x (A(x,y,y)\mconj \forall z B(z,y,t)))^\circ = \exists x (A(x,\xi_1,\xi_2)\mconj \forall z B(z,\xi_3,\xi_4))$. We shall use formulas of the form $A^\circ$ as labels of hyperedges in the canonical model with $\type(A^\circ) = \FVar(A^\circ)$. So, $\mathrm{T} \eqdef \{A^\circ \mid A \in \Fm(\MILLFO(\limpl,\forall))\}$. The idea is that, in the canonical model, variables are represented by nodes, while hyperedge labels do not carry information about variables ($\xi_1,\xi_2,\ldots$ are, informally, placeholders for variables). 
	
	Given a sequent $\Gamma \vdash A$ where $\Gamma = A_1,\ldots,A_n$ ($n \ge 0$) and given two finite sets of variables $X,Y$ such that $Y \subseteq \FVar(\Gamma) \cup X$, let us define the hypergraph $H^{X;Y}_{\Gamma;A} = \langle V, E, \att,\lab,\ext\rangle$ as follows: $V = \FVar(\Gamma) \cup X$;  $E = \{e_1,\ldots,e_n\}$; $\att(e_i)(x) = h_i(x)$ for $x \in \FVar(A_i^\circ)$ where $h_i:\Sigma \to \Sigma$ is a function such that $A_i = A_i^\circ[h_i]$; $\lab(e_i) = A^\circ_i$; $\ext(x) = x$ for $x \in (V \cap \FVar(A)) \cup Y$ (and undefined otherwise).
	
	Let $\mathrm{u}(A) \eqdef \{H^{X;Y}_{\Gamma;A} \mid \Gamma \vdash A~\text{is derivable in $\MILLFO(\limpl,\forall)$}\}$. The proof that $\langle \mathrm{T},\mathrm{u} \rangle$ is a hypergraph language model is quite technical and it can be found in Appendix \ref{appendix:proof-theorem:completeness}.
	Now, assume that $A_1,\ldots,A_n \vdash B$ is true in this model. Then, $K_0 \in \mathrm{u}(A_1 \limpl (\ldots \limpl (A_n \limpl B)))$, i.e. $K_0 = H^{X;Y}_{\emptyset;A_1 \limpl (\ldots \limpl (A_n \limpl B))}$, and $\vdash A_1 \limpl (\ldots \limpl (A_n \limpl B))$ is derivable in $\MILLFO(\limpl,\forall)$. By invertibility of $(\limpl R)$, the sequent $A_1,\ldots,A_n \vdash B$ is derivable.
\end{proof}

\section{Conclusion}\label{section:conclusion}

As we have shown, first-order intuitionistic linear logic does have strong connections to the hypergraph grammar theory, namely, to hypergraph transformation systems and to HR-algrebras. The notion of hypergraph first-order categorial grammars naturally and simply extends the concept of Lambek categorial grammars to hypergraphs. Developing hypergraph $\MILLFO$ grammars and relating them to hypergraph transformation systems gave us useful insights into expressive power of string $\MILLFO$ grammars. In turn, the notion of a hypergraph language model revealed a previously unknown connection of first-order intuitionistic linear logic to the apparatus of HR-algebras, the latter having been studied mainly in the context of monadic second-order definability.

Several questions remain open for the future work. The first one is whether the converse to Theorem \ref{theorem:MILL1G>LTHTS} holds; more generally, it is desirable to characterise precisely hypergraph $\MILLFO$ grammars in terms of hypergraph transformation systems. The second one is whether string $\MILLFO$ grammars over the restricted fragment where each variable is allowed to be used in a formula at most twice generate only nonsemilinear languages (I explained earlier why I think that the answer to this question is negative). The third one is whether $\MILLFO$ is complete w.r.t. hypergraph language models.

\bibliography{hl_fo-logic_Final}

\appendix

\newpage

\section{Proofs}

\subsection{Proof of Proposition \ref{proposition:string-hypergraph-MILL1}}\label{appendix:proof-proposition:string-hypergraph-MILL1}

Let $\Gram=\langle T, S, \triangleright \rangle$ be a string $\MILLFO$ grammar. To make a hypergraph $\MILLFO$ grammar from it, we need to assign some formula to the node symbol. If there was the unit in our calculus, we could just assign it to nodes, but we do not have it. Let us fix a predicate symbol $q$ of arity $0$ not occuring in formulae of $\Gram$. We define the hypergraph $\MILLFO$ grammar $\Gram^\prime = \langle T, S^\prime, \{\lt,\rt\},\triangleright^\prime \rangle$ as follows:
\begin{itemize}
	\item $S^\prime \eqdef S \mconj (q \limpl q)$;
	\item for $a \in T$, $a \triangleright^\prime A$ iff $a \triangleright A$;
	\item $\bullet \triangleright (q \limpl q)$.
\end{itemize}
A string graph $\SG(a_1\ldots a_n)$ is accepted by $\Gram^\prime$ if and only if there is a function $h_E:\{e_1,\ldots,e_n\} \to \Fm(\MILLFO)$ such that $a_i \triangleright h_E(e_i)$ and the sequent
$$
\{h_E(e_i)[x_{i-1}/\lt,x_i/\rt] \mid i=1,\ldots,n\}, (q \limpl q)^{n+1} \vdash S \mconj (q \limpl q)
$$ 
is derivable in $\MILLFO$. By Lemma \ref{lemma:splitting}, this is equivalent to derivability of $\{h_E(e_i)[x_{i-1}/\lt,x_i/\rt] \mid i=1,\ldots,n\} \vdash S $ and to that of $(q \limpl q)^{n+1} \vdash q \limpl q$. The latter sequent is derivable for each $n$, so this is equivalent to the fact that $a_1\ldots a_n$ belongs to $L(\Gram)$.

The other way around, let $\Gram = \langle T, S, \{\lt,\rt\}, \triangleright \rangle$ be a hypergraph $\MILLFO$ grammar. We fix a unary predicate $\mu(x)$ not occurring in $\Gram$ and define the string $\MILLFO$ grammar $\Gram^\prime = \langle T, S^\prime, \triangleright^\prime \rangle$ as follows.
\begin{itemize}
	\item $S^\prime = \mu(\lt) \mconj S$;
	\item $a \triangleright^\prime A$ iff $\type(a) = \{\lt,\rt\}$ and either $A = B \mconj C[\rt/x_\bullet]$ or $A = B \mconj \mu(\lt) \mconj C_1[\lt/x_\bullet]\mconj C_2[\rt/x_\bullet]$ where $a \triangleright B$ and $\bullet \triangleright C$, $\bullet \triangleright C_1$, $\bullet \triangleright C_2$. We call a formula of the form $B \mconj \mu(\lt) \mconj C_1[\lt/x_\bullet]\mconj C_2[\rt/x_\bullet]$ \emph{border} and a formula of the form $B \mconj C[\rt/x_\bullet]$ \emph{non-border}.
\end{itemize}
The grammar $\Gram^\prime$ does not accept the empty string because the sequent $\vdash \mu(x_0) \mconj S[x_0/\lt,x_0/\rt]$ is not derivable in $\MILLFO$. Now, consider a string $w = a_1 \ldots a_n$ for $n>0$. It belongs to $L(\Gram^\prime)$ if and only if there are $A_1,\ldots,A_n$ such that $a_i \triangleright^\prime A_i$ and the sequent $A_1[x_0/\lt,x_1/\rt],\ldots,A_n[x_{n-1}/\lt,x_n/\rt] \vdash \mu(x_0) \mconj S[x_0/\lt,x_n/\rt]$ is derivable in $\MILLFO$. If the latter is the case, then, clearly, $A_1$ is a border formula and, for $i>1$, $A_i$ is a non-border formula. Now, it is not hard to see that this is equivalent to the fact that $\SG(w)$ belongs to $L(\Gram)$. 

\subsection{Proof of Lemma \ref{lemma:main}}\label{appendix:proof-lemma:main}

To analyse derivability of the sequent of interest we shall use Andreoli's triadic calculus $\Sigma_3$ \cite{Andreoli92}. Note that the logic $\ILLFO$ we work with is intuitionistic, while $\Sigma_3$ is developed for classical linear logic. However, a well-known result from \cite{Shellinx91} states that the fragments of (first-order) classical linear logic in the language of $\ILLFO$ that do not include the constant $\mathbf{0}$ are conservative over $\ILLFO$. Thus we can translate intuitionistic formulas into classical ones (in particular, we translate the implication $A \limpl B$ into $A^\bot \parr B$) and analyse the latter.

We use the exact same notation as in \cite{Andreoli92} with the only difference that we switch positive and negative atoms. We do not introduce focusing here in detail but refer the reader to \cite{Andreoli92}. Let us, nevertheless, present the axiom and the inference rules which shall occur in the proof.
$$
\infer[(I_1)]{\Theta : X^\bot \Downarrow X}{}
$$
$$
\infer[(\forall)]{\Theta : \Gamma \Uparrow L, \forall x A}{\Theta : \Gamma \Uparrow L, A[z/x]}
\qquad
\infer[(\exists)]{\Theta : \Gamma \Downarrow \exists x A}{\Theta : \Gamma \Downarrow A[y/x]}
$$
$$
\infer[(\mdisj)]{\Theta : \Gamma \Uparrow L, A \mdisj B}{\Theta : \Gamma \Uparrow L, A,B}
\qquad
\infer[(\mconj)]{\Theta : \Gamma, \Delta \Downarrow A \mconj B}
{\Theta : \Gamma \Downarrow A & \Theta : \Delta \Downarrow B}
$$
$$
\infer[(R\Uparrow)]{\Theta : \Gamma \Uparrow L,C}{\Theta : \Gamma, C \Uparrow L}
\qquad
\infer[(R\Downarrow)]{\Theta : \Gamma \Downarrow D}{\Theta : \Gamma \Uparrow D}
\qquad
\infer[(D_1)]{\Theta : \Gamma, E \Uparrow}{\Theta : \Gamma \Downarrow E}
\qquad
\infer[(D_2)]{\Theta,E : \Gamma \Uparrow}{\Theta,E : \Gamma \Downarrow E}
$$
In the above rules, $X$ is a positive atom; $C$'s outermost connective is not $\mdisj,\forall$; $D$'s outermost connective is not $\mconj,\exists$ and $D$ is not a positive atom; $E$ is not a negative atom. (Recall that positive atoms are formulae of the form $p(x_1,\ldots,x_n)$ and negative atoms are formulae of the form $p^\bot(x_1,\ldots,x_n)$.)

\begin{lemma}\label{lemma:main-appendix}
	\leavevmode
	Let 
	\begin{itemize}
		\item $P,P^\prime$ be multisets of ht-rules; 
		\item $\Theta$ be a multiset of $\LLFO$ formulas and $\Xi$ be a multiset consisting of negative atoms such that, for each $x \in \Var$, $\nu^\bot(x)$ occurs at most once in $\Theta$ and at most once in $\Xi$; 
		\item $h:\ext_G \to \Var$ be a function. 
	\end{itemize}
	\begin{enumerate}
		\item The sequent 
		\begin{equation}\label{equation:sequent-fm-1}
			\vdash \{\fm^\bot(r) \mid r \in P\} : \{\fm^\bot(r) \mid r \in P^\prime\}, \Xi, D(G)[h] \Uparrow
		\end{equation}
		is derivable in $\Sigma_3$ if and only if $h$ is injective, $\Xi=\diag^\bot(G^\prime)$ for some $G^\prime$ such that $\type(G)=\type(G^\prime)$, $h(\ext_G(\sigma)) = \ext_{G^\prime}(\sigma)$ for $\sigma \in \type(G)$ and such that there is a derivation of a hypergraph isomorphic to $G^\prime$ from $G$ which uses each rule from $P^\prime$ exactly once and that can use rules from $P$ any number of times.
		\item The sequent 
		\begin{equation}\label{equation:sequent-fm-2}
			\vdash \{\fm^\bot(r) \mid r \in P\} : \Theta \Downarrow D(G)[h]
		\end{equation}
		is derivable in $\Sigma_3$ if and only if $h$ is injective and $\Theta=\diag^\bot(G^\prime)$ for some $G^\prime$ isomorphic to $G$ such that $h(\ext_G(\sigma)) = \ext_{G^\prime}(\sigma)$ for $\sigma \in \type(G)$.
	\end{enumerate}
\end{lemma}
\begin{proof}
	Let us denote $\{\fm^\bot(r) \mid r \in P\}$ by $\Phi$ and $\{\fm^\bot(r) \mid r \in P^\prime\}$ by $\Phi^\prime$.
	
	Both statements of the lemma are proved jointly by induction on the length of a derivation. Let us prove their ``only if'' parts. Essentially, we need to do a proof search. Let us start with considering possible cases of the last rule applied in a derivation of (\ref{equation:sequent-fm-1}).
	
	\textit{Case 1.}
	$$
	\infer[(D_1)]
	{
		\vdash \Phi : \Phi^\prime, \Xi, D(G)[h] \Uparrow
	}
	{
		\vdash \Phi : \Phi^\prime, \Xi \Downarrow D(G)[h]
	}
	$$
	
	The premise is a sequent of the form (\ref{equation:sequent-fm-2}). The induction hypothesis completes the proof.
	
	\textit{Case 2.}
	$$
	\infer[(D_1)]
	{	\vdash \Phi : \{\fm^\bot(r) \mid r \in P^\prime\}, \Xi, D(G)[h] \Uparrow	}
	{	\vdash \Phi : \{\fm^\bot(r) \mid r \in P^{\prime\prime}\}, \Xi, D(G)[h] \Downarrow \fm^\bot(p)		}
	$$
	Here $p = (H_1 \to H_2) \in P^\prime$ and $P^{\prime\prime} = P^\prime \setminus \{p\}$. Let us denote $\{\fm^\bot(r) \mid r \in P^{\prime\prime}\}$ by $\Phi^{\prime\prime}$. It holds that $\fm^\bot(p) = \exists \vec{u} (D(H_2) \mconj D^\bot(H_1)[\chi_p])$ where $\vec{u}$ is the list of nodes in $\ran(\ext_{H_2})$; therefore, the last steps in the derivation of the above premise must be applications of $(\exists)$:
	$$
	\infer[(\exists)^\ast]
	{	\vdash \Phi : \Phi^{\prime\prime}, \Xi, D(G)[h] \Downarrow \fm^\bot(p)		}
	{	\vdash \Phi : \Phi^{\prime\prime}, \Xi, D(G)[h] \Downarrow (D(H_2) \mconj D^\bot(H_1)[\chi_p])[f]	}
	$$
	(Here $(\exists)^\ast$ means that $(\exists)$ is applied several times.) The function $f$ maps $\ran(\ext_{H_2})$ to $\Var$. Now, there are two options how the premise of the above rule application can be obtained.
	
	\textit{Case 2a.}
	$$
	\infer[(\mconj)]
	{	\vdash \Phi : \Phi^{\prime\prime}, \Xi, D(G)[h] \Downarrow (D(H_2) \mconj D^\bot(H_1)[\chi_p])[f]	}
	{	\vdash \Phi : \Phi^{\prime\prime}_1, \Xi_1, D(G)[h] \Downarrow D(H_2)[f] 
		&
		\vdash \Phi : \Phi^{\prime\prime}_2, \Xi_2 \Downarrow D^\bot(H_1)[\chi_p])[f]	}
	$$
	Here $\Phi_1^{\prime\prime},\Phi_2^{\prime\prime}=\Phi^{\prime\prime}$ and $\Xi_1,\Xi_2=\Xi$. By the induction hypothesis applied to the leftmost premise (which is of the form (\ref{equation:sequent-fm-2})), it must be the case that $\Xi_1, D(G)[h]$ consists of negative atoms. However, $D(G)[h]$ is not an atom, so this case is impossible. 
	
	\textit{Case 2b.}
	$$
	\infer[(\mconj)]
	{	\vdash \Phi : \Phi^{\prime\prime}, \Xi, D(G)[h] \Downarrow (D(H_2) \mconj D^\bot(H_1)[\chi_p])[f]	}
	{	\vdash \Phi : \Phi^{\prime\prime}_1, \Xi_1 \Downarrow D(H_2)[f] 		&		\vdash \Phi : \Phi^{\prime\prime}_2, \Xi_2, D(G)[h] \Downarrow D^\bot(H_1)[\chi_p][f]	}
	$$
	By the induction hypothesis, $f$ is injective, $\Phi_1^{\prime\prime}=\emptyset$ and $\Xi_1 = \diag^\bot(H_2^\prime)$ where $H_2^\prime$ is isomorphic to $H_2$ and $f(\ext_{H_2}(\sigma)) = \ext_{H_2^\prime}(\sigma)$ for $\sigma \in \type(H_2)$. Consequently, $\Phi_2^{\prime\prime} = \Phi^{\prime\prime}$.
	
	Now, let us analyse possible derivations of the rightmost premise. Let $f^\prime \eqdef f \circ \chi_p$. The formula $D^\bot(H_1)[f^\prime]$ equals $\forall \vec{v} \bigmdisj \diag^\bot(H_1)[f^\prime]$ where $\vec{v}$ is the list of nodes in $V_{H_1} \setminus \ran(\ext_{H_1})$. Any derivation of the rightmost premise must end in the following way:
	$$
	\infer[(R \Downarrow)]
	{	
		\vdash \Phi : \Phi^{\prime\prime}, \Xi_2, D(G)[h] \Downarrow D^\bot(H_1)[f^\prime]
	}
	{
		\infer[(\forall)^\ast]
		{	
			\vdash \Phi : \Phi^{\prime\prime}, \Xi_2, D(G)[h] \Uparrow D^\bot(H_1)[f^\prime]
		}
		{
			\infer[(\mdisj)^\ast]
			{	
				\vdash \Phi : \Phi^{\prime\prime}, \Xi_2, D(G)[h] \Uparrow \bigmdisj \diag^\bot(H_1)[f^\prime][g]
			}
			{
				\infer[(R \Uparrow)^\ast]
				{	
					\vdash \Phi : \Phi^{\prime\prime}, \Xi_2, D(G)[h] \Uparrow \diag^\bot(H_1)[f^\prime][g]
				}
				{
					\vdash \Phi : \Phi^{\prime\prime}, \Xi_2, \diag^\bot(H_1)[f^\prime][g], D(G)[h] \Uparrow 
				}
			}
		}
	}
	$$
	Here $g:V_{H_1} \setminus \ran(\ext_{H_1}) \to \Var$ is a substitution such that, for each $v \in V_{H_1} \setminus \ran(\ext_{H_1})$, $g(v)$ is a fresh variable, i.e. one not occuring in $\Xi_2,D(G)[h],\diag^\bot(H_1)[f^\prime]$ (for $v \in \dom(g)$). 
	
	Let $\xi$ be a variable. We shall prove now that $\nu(\xi)$ cannot occur twice in $\Xi_2,\diag^\bot(H_1)[f^\prime][g]$. The formula $\nu^\bot(\xi)$ cannot occur in $\Xi_2$ twice because it occurs in $\Xi$ at most once. It cannot occur in $\diag^\bot(H_1)[f^\prime][g]$ twice as well. Indeed, no formula of the form $\nu^\bot(x)$ occurs twice in $\diag^\bot(H_1)$ (see Definition \ref{definition:diagram}). The functions $f$ and $\chi_p$ are injective, so the same property holds for $\diag^\bot(H_1)[f^\prime]$. Finally, the function $g$ is also injective because it replaces some variables in $\diag^\bot(H_1)[f^\prime]$ by fresh ones. Therefore, the desired property holds for the multiset $\diag^\bot(H_1)[f^\prime][g]$ as well.
	
	Finally, assume, ex falso, that $\nu^\bot(\xi)$ occurs once in $\Xi_2$ and once in $\diag^\bot(H_1)[f^\prime][g]$. It is not the case that $\xi \in \ran(g)$ because $g$ introduces only fresh variables. Therefore, $\xi = f^\prime(v)=f(\chi_p(v))$ for $v = \ext_{H_1}(\sigma)$, $\sigma \in \type(H_1)$. Consequently, $\xi = f(\ext_{H_2}(\sigma))=\ext_{H_2^\prime}(\sigma)$. However, $\nu^\bot(\ext_{H_2^\prime}(\sigma))$ belongs to $\Xi_1 = \diag^\bot(H_2^\prime)$, so it occurs in $\Xi=\Xi_1,\Xi_2$ twice. This is a contradiction.
	
	Summing up, one can apply the induction hypothesis to the sequent 
	$$
	\vdash \Phi : \Phi^{\prime\prime}, \Xi_2, \diag^\bot(H_1)[f^\prime][g], D(G)[h] \Uparrow$$
	and conclude that
	\begin{itemize}
		\item $h$ is injective;
		\item $\Xi_2,\diag^\bot(H_1)[f^\prime][g] = \diag^\bot(H^\prime)$ for a hypergraph $H^\prime$ such that $\type(H^\prime)=\type(G)$, $h(\ext_G(\sigma))=\ext_{H^\prime}(\sigma)$ for $\sigma \in \type(G)$;
		\item there is a derivation of a hypergraph isomorphic to $H^\prime$ from $G$ that uses each rule from $P^{\prime\prime}$ exactly once and that can use rules from $P$ arbitrary many times.
	\end{itemize}
	Since $g \circ f^\prime$ is injective, it holds that $\diag^\bot(H_1)[f^\prime][g]=\diag^\bot(H_1^\prime)$ where $H_1^\prime$ is isomorphic to $H_1$. Let $K$ be a hypergraph obtained from $H^\prime$ by removing nodes belonging to $\ran(g)$, removing hyperedges belonging to $E_{H^\prime_1}$, and adding a new hyperedge $e_0$ such that $\type_K(e_0) = \type(H_1)$ and $\att_K(e_0)(\sigma)=\ext_{H^\prime_1}(\sigma)$. Clearly, $K[e_0/H_1]$ is isomorphic to $H^\prime$. 
	
	The multiset $\Xi_2$ equals
	$$
	\Xi_2 = \{\lab_{K}(e)[\att_K(e)] \mid e \in E_K \setminus \{e_0\}\} \cup \{\nu(v) \mid v \in V_K \setminus \ran(\att_K(e_0))\}.
	$$
	Earlier, we proved that $\Xi_1 = \diag^\bot(H_2^\prime)$ for $H_2^\prime$ being isomorphic to $H_2$. For $\sigma \in \type(H_2^\prime)$, it holds that $\ext_{H_2^\prime}(\sigma) = f(\ext_{H_2}(\sigma)) = f(\chi_p(\ext_{H_1}(\sigma))) = f^\prime(\ext_{H_1}(\sigma)) = \ext_{H_1^\prime}(\sigma)$. 
	
	Assume that $v \in V_{H_2^\prime} \setminus \ran(\ext_{H_2^\prime})$. Then, $\nu(v)$ belongs to $\Xi_1$. If $v$ appears as a free variable in $\Xi_2$, then, since $v\notin\ran(\ext_{H_2^\prime})=\ran(\att_K(e_0))$, there must be the formula $\nu(v)$ in $\Xi_2$. But this would imply that $\nu(v)$ occurs twice in $\Xi$, which is a contradiction. Therefore, $v$ does not appear in $\Xi_2$. All this implies that $\Xi_1,\Xi_2 = \diag^\bot(G^\prime)$ where $G^\prime$ is isomorphic to $K[e_0/H_2^\prime]=K[e_0/H_2]$. Summing up, $G^\prime$ is derivable from $G$ using rules from $P$, using exactly once rules from $P^{\prime\prime}$ and the rule $p=(H_1 \to H_2)$.
	
	The only thing remaining to be proved is that $h(\ext_G(\sigma)) = \ext_{G^\prime}(\sigma)$. We know that $h(\ext_G(\sigma)) = \ext_{H^\prime}(\sigma)$. The procedure described above that transforms $H^\prime$ into $G^\prime$ does not affect external nodes. Therefore, $\ext_{H^\prime}(\sigma)=\ext_{G^\prime}(\sigma)$, as desired.
 	
	\textit{Case 3.}
	$$
	\infer[(D_1)]
	{	\vdash \Phi : \Phi^\prime, \Xi, D(G)[h] \Uparrow	}
	{	\vdash \Phi : \Phi^\prime, \Xi, D(G)[h] \Downarrow \fm^\bot(p)		}
	$$
	Here $\fm^\bot(p) \in \Phi$. This case is analogous to Case 2.
	
	Now, let us proceed to the second statement of the Lemma. Assume that the sequent $\vdash \{\fm^\bot(r) \mid r \in P\} : \Theta \Downarrow D(G)[h]$ is derivable. The formula $D(G)[h]$ is of the form $\exists \vec{v} \bigotimes \diag(G)[h]$; thus, any derivation of the sequent must end as follows:
	$$
	\infer[(\exists)^\ast]
	{
		\vdash \{\fm^\bot(r) \mid r \in P\} : \Theta \Downarrow D(G)[h]
	}
	{
		\vdash \{\fm^\bot(r) \mid r \in P\} : \Theta \Downarrow \bigotimes \diag(G)[h][g]
	}
	$$
	Here $g$ maps nodes from $V_G \setminus \ran(\ext_G)$ to variables. The premise can only be obtained by consecutive applications of $(\mconj)$. Going in this proof from bottom to top, one finally reaches sequents of the form $\vdash \{\fm^\bot(r) \mid r \in P\} : \tilde{\Theta} \Downarrow X$ where $\tilde{\Theta}$ is a subset of $\Theta$ and $X$ is an atom from $\diag(G)[h][g]$. Such a sequent can only be an axiom, so $\tilde{\Theta}=X^\bot$. Going back from top to bottom, we conclude that $\Theta=\diag^\bot(G)[h][g]$.
	
	The functions $h$ and $g$ have disjoint domains such that $\dom(h)\cup\dom(g)=V_G$. Let us unify these functions into a single one $f$. If $f$ is not injective, e.g. $f(v_1)=f(v_2)=x$, then $\Theta$ contains $\nu(x)$ twice, which is a contradiction. Thus, $f$ is injective. This implies that $h$ is injective and that $\Theta=\diag^\bot(G^\prime)$ where $G^\prime$ is isomorphic to $G$ via the node isomorphism $f$; besides, $h(\ext_G(\sigma))=\ext_{G^\prime}(\sigma)$. 
	
	The above reasonings, in fact, also prove the ``if'' parts of Lemma \ref{lemma:main-appendix}, because they show how to construct derivations of sequents of interest.
\end{proof}

Lemma \ref{lemma:main} follows directly from Lemma \ref{lemma:main-appendix}. Indeed, the sequent 
$$
\{\bang \fm(r) \mid r \in P \}, \{\fm(r) \mid r \in P^\prime \}, \diag(G^\prime) \vdash D(G)[\chi_{G \to G^\prime}]
$$ 
is derivable in $\ILLFO$ if and only if its translation into $\LLFO$
$$
\vdash \{\fm^\bot(r) \mid r \in P \}: \{\fm^\bot(r) \mid r \in P^\prime \}, \diag^\bot(G^\prime) , D(G)[\chi_{G \to G^\prime}] \Uparrow
$$
is derivable in $\Sigma_3$ \cite[Theorems 1, 2]{Andreoli92}. This, according to Lemma \ref{lemma:main-appendix}, is equivalent to the fact that a hypergraph isomorphic to $G^\prime$ is derivable from $G$ using each rule from $P^\prime$ once and using some rules from $P$.

\subsection{Alternative Proof of Theorem \ref{theorem:MILL1G-NPC}}\label{appendix:proof-theorem:MILL1G-NPC}

	In view of Corollary \ref{corollary:string-MILL1-grammars}, it suffices to present a linear-time ht-system generating an NP-complete language of string graphs. Let $P,Q_1,Q_2,S,T_1,T_2,U$ be nonterminal symbols such that $\type(Q_1)=\type(Q_2)=\emptyset$, $\type(P) = \type(S)=\type(T_1)=\type(U)=\{\lt,\rt\}$ and $\type(T_2)=\{1,2,3,4\}$; let $0,1,a,b,c$ be terminal symbols with type $\{\lt,\rt\}$. Let the start hypergraph be $S^\bullet+Q_1^\bullet$ (`$+$' is introduced in Definition \ref{definition:disjoint-union}). The rules are presented below.
	\begin{enumerate}
		\item $S^\bullet+Q_1^\bullet \to \SG(SaT_1aT_1aT_1)+Q_1^\bullet$;
		\item $T_1^\bullet+Q_1^\bullet \to \SG(kT_1)+Q_1^\bullet$, \; $T_1^\bullet+Q_1^\bullet \to \SG(k)+Q_1^\bullet$ for $k=0,1$;
		\item 
		$
		S^\bullet+Q_1^\bullet \to 
		\vcenter{\hbox{{\tikz[baseline=.1ex]{
						\def\HOR{1}
						\def\VER{1}
						\def\BEND{30}
						\foreach \l in {1,2,3}{
							\node[hyperedge] (E\l) at ($({\VER*(2*\l+0.5)},\HOR*0.5)$) {$T_2$};
						}
						\node[hyperedge] (E4) at ($({\VER*8},\HOR*0.5)$) {$Q_1$};
						\foreach \i in {0,...,7}
						{
							\node[node, label=\ifnumequal{\i}{0}{left}{right}:{\tiny \ifnumequal{\i}{0}{$(\lt)$}{\ifnumequal{\i}{7}{$(\rt)$}{}}}] (V\i0) at ($(\VER*\i,0)$) {};
							\ifnumgreater{\i}{1}{
								\node[node] (V\i1) at ($(\VER*\i,\HOR)$) {};	
							}{}
						}
						\foreach \k in {1,2,3}{
							\pgfmathtruncatemacro\L{2*\k}
							\pgfmathtruncatemacro\R{2*\k+1}
							\draw[thick,-latex] (V\L1) -- node[above] {$P$} (V\R1);
							\draw[-] (E\k) -- node[right] {\tiny $1$} (V\L1);
							\draw[-] (E\k) -- node[below] {\tiny $2$} (V\R1);
							\draw[-] (E\k) -- node[above] {\tiny $3$} (V\L0);
							\draw[-] (E\k) -- node[left] {\tiny $4$} (V\R0);
						}
						\foreach \k in {1,2,3}{
							\pgfmathtruncatemacro\L{2*\k-1}
							\pgfmathtruncatemacro\R{2*\k}
							\draw[thick,-latex] (V\L0) -- node[below] {$a$} (V\R0);
						}
						\draw[thick,-latex] (V00) -- node[below] {$S$} (V10);
		}}}};
		$
		\item
		$
		T_2^\bullet + Q_1^\bullet \to 
		\vcenter{\hbox{{\tikz[baseline=.1ex]{
						\def\HOR{1}
						\def\VER{1}
						\def\BEND{30}
						\node[hyperedge] (E1) at ($({\VER*1.5},\HOR*0.5)$) {$T_2$};
						\node[hyperedge] (E2) at ($({\VER*3},\HOR*0.5)$) {$Q_1$};
						\foreach \i in {0,...,2}
						{
							\foreach \j in {0,1}{
								\pgfmathtruncatemacro\NUM{3-(2*\j)+(\i/2)}
								\node[node, label=\ifnumequal{\i}{0}{left}{right}:{\tiny \ifnumequal{\i}{0}{$(\NUM)$}{\ifnumequal{\i}{2}{$(\NUM)$}{}}}] (V\i\j) at ($(\VER*\i,\HOR*\j)$) {};
							}
						}
						\draw[-] (E1) -- node[right] {\tiny $1$} (V11);
						\draw[-] (E1) -- node[below] {\tiny $2$} (V21);
						\draw[-] (E1) -- node[above] {\tiny $3$} (V10);
						\draw[-] (E1) -- node[left] {\tiny $4$} (V20);
						\draw[thick,-latex] (V01) -- node[above] {$k$} (V11);
						\draw[thick,-latex] (V00) -- node[below] {$k$} (V10);
		}}}}
		$, \;
		$
		T_2^\bullet + Q_1^\bullet \to 
		\vcenter{\hbox{{\tikz[baseline=.1ex]{
						\def\HOR{1}
						\def\VER{1}
						\def\BEND{30}
						\node[hyperedge] (E1) at ($({\VER*2},\HOR*0.5)$) {$Q_1$};
						\foreach \i in {0,1}
						{
							\foreach \j in {0,1}{
								\pgfmathtruncatemacro\NUM{3-(2*\j)+\i}
								\node[node, label=\ifnumequal{\i}{0}{left}{right}:{\tiny \ifnumequal{\i}{0}{$(\NUM)$}{\ifnumequal{\i}{1}{$(\NUM)$}{}}}] (V\i\j) at ($(\VER*\i,\HOR*\j)$) {};
							}
						}
						\draw[thick,-latex] (V01) -- node[above] {$k$} (V11);
						\draw[thick,-latex] (V00) -- node[below] {$k$} (V10);
		}}}}
		$ for $k=0,1$;
		\item $S^\bullet+Q_1^\bullet \to U^\bullet + Q_2^\bullet$;
		\item 
		$
		\vcenter{\hbox{{\tikz[baseline=.1ex]{
						\def\HOR{1}
						\def\VER{1}
						\def\BEND{30}
						\node[hyperedge] (E1) at ($({\VER*2},\HOR*0.5)$) {$Q_2$};
						\foreach \i in {0,1}
						{
							\foreach \j in {0,1}{
								\pgfmathtruncatemacro\NUM{3-(2*\j)+\i}
								\node[node, label=\ifnumequal{\i}{0}{left}{right}:{\tiny \ifnumequal{\i}{0}{$(\NUM)$}{\ifnumequal{\i}{1}{$(\NUM)$}{}}}] (V\i\j) at ($(\VER*\i,\HOR*\j)$) {};
							}
						}
						\draw[thick,-latex] (V01) -- node[above] {$P$} (V11);
						\draw[thick,-latex] (V00) -- node[below] {$U$} (V10);
		}}}}
		\;\to
		\vcenter{\hbox{{\tikz[baseline=.1ex]{
						\def\HOR{1}
						\def\VER{1}
						\def\BEND{30}
						\node[hyperedge] (E1) at ($({\VER*3},\HOR*0.5)$) {$Q_2$};
						\foreach \i in {0,1}
						{
							\foreach \j in {0,1}{
								\pgfmathtruncatemacro\NUM{3-(2*\j)+\i}
								\node[node, label=\ifnumequal{\i}{0}{left}{right}:{\tiny \ifnumequal{\i}{0}{$(\NUM)$}{\ifnumequal{\NUM}{2}{$(\NUM)$}{}}}] (V\i\j) at ($(\VER*\i,\HOR*\j)$) {};
							}
						}
						\node[node, label=right:{\tiny $(4)$}] (V20) at ($(\VER*2,\HOR*0)$) {};
						\draw[thick,-latex] (V00) -- node[left] {$b$} (V01);
						\draw[thick,latex-] (V10) -- node[right] {$b$} (V11);
						\draw[thick,-latex] (V10) -- node[below] {$U$} (V20);
		}}}};
		$
		\item $U^\bullet + Q_2^\bullet \to c^\bullet$.
	\end{enumerate}
	
	This ht-system is linear-time because each production, except for 5, which is applied at most once in any derivation, increases the number of terminal symbols. This ht-system generates hypergraphs of the form
	\begin{equation}\label{equation:NP-complete}
		\SG(bw_1b bw_2b \ldots bw_m b c a x_1 a y_1 a z_1 \ldots a x_n a y_n a z_n)
	\end{equation}
	such that $w_i,x_i,y_i,z_i$ are nonempty strings over the alphabet $\{0,1\}$ and such that $\{w_1,\ldots,w_m\}$ coincides as a multiset with $\{x_{i_1},y_{i_1},z_{i_1},\ldots,x_{i_l},y_{i_l},z_{i_l}\}$ for some $1 \le i_1 < \ldots < i_l \le n$. Consequently, one can reduce the exact cover problem by 3-sets to this language: if $X=\{w_1,\ldots,w_{m}\}$ is a set and $C = \{\{x_1,y_1,z_1\},\ldots,\{x_n,y_n,z_n\}\}$ is a collection of 3-element subsets of $X$, then checking whether $X$ is the disjoint union of some sets from $C$ is equivalent to checking whether the hypergraph (\ref{equation:NP-complete}) is generated by the ht-system.

\subsection{Proof of Theorem \ref{theorem:completeness}}\label{appendix:proof-theorem:completeness}

\begin{lemma}\label{lemma:properties-HXYGA}
	\leavevmode
	\begin{enumerate}
		\item Hypergraphs $H_1 = H^{X;Y_1}_{\Gamma;A_1}$ and $H_2 = H^{X;Y_2}_{\Gamma;A_2}$ are equal (as concrete hypergraphs) if and only if $\type(H_1) = \type(H_2)$.
		\item If $v$ is an isolated node in $H = H^{X;Y}_{\Gamma;A}$, then $H^\prime = H^{X\setminus\{v\};Y\setminus\{v\}}_{\Gamma;A}$ is obtained from $H$ by removing $v$.
	\end{enumerate}
\end{lemma}
\begin{proof}
	\begin{enumerate}
		\item This follows from the definition of $H = H^{X;Y}_{\Gamma;A}$, because $Y,A$ define only the domain of $\ext_H$.
		\item The node $v$ being isolated is equivalent to the fact that $v \notin \FVar(\Gamma)$. Therefore, $V_{H^\prime} = \FVar(\Gamma) \cup (X\setminus\{v\}) =  (\FVar(\Gamma) \cup X) \setminus \{v\} = V_H \setminus\{v\}$. Clearly, only the node $v$ is affected, so the statement follows.
	\end{enumerate}
	
\end{proof}

We start with showing that $\sub_h(\mathrm{u}(A)) \subseteq \mathrm{u}(A[h])$. Clearly, it suffices to check this property only for the case where $h$ changes only one variable, e.g. $h(z) = t$ ($z \ne t$) and $h(x)=x$ for $x \ne z$. Indeed, $\sub_h(\sub_g(H)) = \sub_{h \circ g}(H)$ and $\sub_h$ is a monotone operation.

Let $H  = H^{X;Y}_{\Gamma; A}$ for $\Gamma \vdash A$ being derivable and let $h(z) = t$ ($z \ne t$) and $h(x)=x$ for $x \ne z$. Our goal is to show that a hypergraph isomorphic to $\sub_h(H)$ belongs to $\mathrm{u}(A[h])$.

\textit{Case 1.} $z \notin \type(H)$. Then $\sub_h(H)=H$. If $z \notin \FVar(A)$, then $A[t/z]=A$, so $H \in \mathrm{u}(A[t/z])$, and we are done. 

Assume that $z \in \FVar(A)$; then, $z \notin (\FVar(\Gamma) \cup X)$. Consequently, $\Gamma \vdash A[t/z]$ is derivable because $z$ does not occur freely in $\Gamma$.
Let us define $H^\prime \eqdef H^{X;Y}_{\Gamma; A[t/z]} \in \mathrm{u}(A[t/z])$. $H^\prime$ and $H$ can possibly differ only in $\type(H)$ and $\type(H^\prime)$; if $\type(H)=\type(H^\prime)$, then $H=H^\prime$ (Lemma \ref{lemma:properties-HXYGA}), and the proof is completed. Since $z \notin \type(H)$, $z \notin \type(H^\prime)$. Besides, if $t \in \type(H)$, then $t \in \type(H^\prime)$. 

Assume that $t \in \type(H^\prime) \setminus \type(H)$ (this is the only remaining possibility for $\type(H)$ and $\type(H^\prime)$ to differ). This implies that $t \notin Y$ and that $t \notin (\FVar(\Gamma) \cup X) \cap \FVar(A)$. Let $\tau$ be a fresh variable not occuring in $\Gamma,A,X,Y$. Then, $\hat{H} \eqdef H^{X[\tau/t];Y}_{\Gamma[\tau/t];A}$ is isomorphic to $H$ and $\Gamma[\tau/t] \vdash A$ is derivable. Indeed, either $t$ is not in $\FVar(\Gamma) \cup X$, and hence substitution changes nothing; or $t$ is not free in $A$, hence $\hat{H}$ is obtained from $H$ by just renaming the node $t$ as $\tau$. External nodes are not changed because $t$ is not external in $H$ and the new node $\tau$ is not external in $\hat{H}$ because it does not belong to $\FVar(A)$ or to $Y$. 

The sequent $\Gamma[\tau/t] \vdash A[t/z]$ is derivable because $z$ does not occur freely in $\Gamma[\tau/t]$. Consider $\hat{H}^\prime = H^{X[\tau/t];Y}_{\Gamma[\tau/t];A[t/z]}$. It holds that $t \notin \type(\hat{H}^\prime)$ as well as $t \notin \type(\hat{H})$. Thus, $\type(\hat{H}) = \type(\hat{H}^\prime)$, so $\hat{H}=\hat{H}^\prime \in \mathrm{u}(A[t/z])$. Therefore, $H$ belongs to $\mathrm{u}(A[t/z])$, as we consider membership to a hypergraph language up to isomorphism.

\textit{Case 2.} $z \in \type(H)$, $t \notin \type(H)$. Let $\tau$ be a fresh variable; define a function $g$ such that $g(z)=t$, $g(t)=\tau$, otherwise identical. The hypergraph $H$ is isomorphic to $H_1 \eqdef H^{X[\tau/t];Y[\tau/t]}_{\Gamma[\tau/t];A[\tau/t]}$. Indeed, $H_1$ is obtained from $H$ by simply renaming $t$ by $\tau$ and, since $t$ is not external, external nodes of $H$ remain unchanged. Now, note that $\sub_h(H_1)$ is isomorphic to $H_2 \eqdef H^{X[g];Y[g]}_{\Gamma[g];A[g]}$. Indeed, $\sub_h(H_1)$ is a renaming of the external node $z$ by $t$; in turn, $H_2$ is obtained from $H_1$ by changing the node $z$ to $t$ (with $\ext_{H_2}(t)=t$).

Finally, $H_2$ equals $H_3 \eqdef H^{X[g];Y[g]}_{\Gamma[g];A[t/z]}$. This can be verified by Lemma \ref{lemma:properties-HXYGA}. Indeed, both $\type(H_2)$ and $\type(H_3)$ do not contain $z$ because $z$ does not belong even to $V_{H_2}$ and to $V_{H_3}$. Besides, $t \in \type(H_2)$ and $t \in \type(H_3)$ because $z \in \type(H_1)$; finally, $\tau \notin \type(H_2)$. Thus, $\type(H_2) = \type(H_3)$, hence $H_2=H_3 \in \mathrm{u}(A[t/z])$. We have proved that a hypergraph isomorphic to $\sub_h(H)$ belongs to $\mathrm{u}(A[t/z])$, as desired.

\textit{Case 3.} $z,t \in \type(H)$. In this case, take $H^\prime \eqdef H^{X[h];Y[h]}_{\Gamma[h];A[h]}$. $H^\prime$ is obtained from $H$ by identifying the node $z$ with the node $t$, the resulting node being named $t$. Thus, clearly, $H^\prime$ is isomorphic to $\sub_h(H)$.

Let us check that $\mathrm{u}(\forall x A) = \bigcap\limits_{y \in \Var} \mathrm{u}(A[y/x])$. Take $H = H^{X;Y}_{\Gamma;\forall x A} \in \mathrm{u}(\forall x A)$. 
\begin{itemize}
	\item If $y \notin V_H = \FVar(\Gamma) \cup X$, then $H^{X;Y}_{\Gamma; \forall x A} = H^{X;Y}_{\Gamma;A[y/x]}$, and, since $\Gamma \vdash A[y/x]$ is derivable, it holds that $H \in \mathrm{u}(A[y/x])$. 
	\item If $y \in \ran(\ext_H) = (V_H \cap \FVar(\forall x A)) \cup Y$, then it is also the case that $H^{X;Y}_{\Gamma; \forall x A} = H^{X;Y}_{\Gamma;A[y/x]}$ and we are done.
	\item Let $y \in V_H \setminus (\FVar(\forall x A) \cup Y)$. Let $z$ be a fresh variable and let $\Gamma^\prime$ and $X^\prime$ be obtained by replacing $y$ by $z$ in $\Gamma$ and $X$ resp. Then, $H = H^{X;Y}_{\Gamma;\forall x A}$ is isomorphic to $H_{\Gamma^\prime;\forall x A}^{X^\prime;Y}$ and $\Gamma^\prime \vdash \forall x A$ is derivable (since $z$ is not free in $\forall x A$); in turn, $H_{\Gamma^\prime;\forall x A}^{X^\prime;Y} = H_{\Gamma^\prime;A[y/x]}^{X^\prime;Y}$ because $y \notin \FVar(\Gamma^\prime) \cup X^\prime$. This implies that $H$ (as an abstract hypergraph) belongs to $\mathrm{u}(A[y/x])$, as desired. Note that $\Gamma^\prime \vdash A[y/x]$ is derivable because of invertibility of $(\forall R)$. 
\end{itemize}
Thus, $\mathrm{u}(\forall x A) \subseteq \bigcap\limits_{y \in \Var} \mathrm{u}(A[y/x])$. Let us show the converse inclusion. Let $H \in \bigcap\limits_{y \in \Var} \mathrm{u}(A[y/x])$ and let $z$ be a fresh variable not in $\type(H)$. Then, for some $X,Y,\Gamma$, it holds that $H = H^{X;Y}_{\Gamma;A[z/x]}$ and that $\Gamma \vdash A[z/x]$ is derivable. Note that $\type(H) = ((\FVar(\Gamma) \cup X) \cap \FVar(A[z/x])) \cup Y$; therefore, $z \notin \FVar(\Gamma) \cup X$ and, by $(\forall R)$, the sequent $\Gamma \vdash \forall x A$ is derivable. It remains to observe that $H^{X;Y}_{\Gamma;A[z/x]} = H^{X;Y}_{\Gamma;\forall x A}$.

Let us check that $\mathrm{u}(A \limpl B) \subseteq \{H \mid \forall H^\prime \in \mathrm{u}(A) \, (H \pc H^\prime \in \mathrm{u}(B))\}$. Take hypergraphs $H_1 = H_{\Gamma;A \limpl B}^{X;Y} \in \mathrm{u}(A \limpl B)$ and $H_2 = H_{\Delta;A}^{Z;W} \in \mathrm{u}(A)$. Without loss of generality, let us assume that variables in $(\FVar(\Gamma) \cup X) \setminus \type(H_1)$ and $(\FVar(\Delta) \cup Z) \setminus \type(H_2)$ are replaced by fresh ones (such a replacement results in isomorphic hypergraphs). Nodes in $H_1$ are variables from $\FVar(\Gamma) \cup X$ and nodes in $H_2$ are variables from $\FVar(\Delta) \cup Z$.

Let $H_3 \eqdef H_{\Gamma , \Delta; B}^{X \cup Z ; Y \cup W \cup U}$ where $U = (\FVar(\Gamma, \Delta) \cup X \cup Z) \cap \FVar(A)$. We claim that $H_1 \pc H_2 = H_3$. First, note that the above variable freshness condition guarantees that only common external nodes of $H_1$ and $H_2$ are identified in $H_3$. Let us also check that external nodes in $H_1\pc H_2$ and in $H_3$ are the same, i.e. that $\type(H_1 \pc H_2) = \type(H_3)$. 
\begin{align*}
	\type(H_1 \pc H_2) &= \type(H_1) \cup \type(H_2) 
	\\
	& = ((\FVar(\Gamma) \cup X) \cap \FVar(A \limpl B)) \cup Y \cup
	\\ & \hspace{3cm} ((\FVar(\Delta) \cup Z) \cap \FVar(A)) \cup W
	\\
	& \subseteq
	((\FVar(\Gamma , \Delta) \cup X \cup Z) \cap \FVar(B)) \cup Y \cup W \cup U = \type(H_3).
\end{align*}
To check that $\type(H_1 \pc H_2) \supseteq \type(H_3)$, assume that $x \in \type(H_3) \setminus (\type(H_1) \cup \type(H_2))$. This implies that $x$ belongs to $(\FVar(\Delta) \cup Z) \cap \FVar(B)$. However, this contradicts the fact that $x$ is fresh for $B$. Thus, $\type(H_3) = \type(H_1) \pc \type(H_2)$.

To show that $\mathrm{u}(A \limpl B) \supseteq \{H \mid \forall H^\prime \in \mathrm{u}(A) \, (H \pc H^\prime \in \mathrm{u}(B))\}$, assume that we are given a hypergraph $H$ such that $H \pc H^\prime$ belongs to $\mathrm{u}(B)$ for each $H^\prime \in \mathrm{u}(A)$. Take $H^\prime \eqdef H^{\type(H);\type(H)}_{A;A} \in \mathrm{u}(A)$. Then $H \pc H^\prime = H^{X;Y}_{\Gamma;B}$ for some $X,Y,\Gamma$. Clearly,
\begin{itemize}
	\item $\Gamma = A, \Gamma^\prime$;
	\item $\type(H \pc H^\prime) = \type(H) \cup \FVar(A)$;
	\item $H$ is obtained from $H^{X;Y}_{\Gamma;B}$ by removing one $A$-labeled hyperedge and then external nodes from $\FVar(A) \setminus \type(H)$. To be more precise, if $x\in \FVar(A) \setminus \type(H)$, then $x \in \type(H \pc H^\prime)$ but $x \notin \type(H)$; thus, after removing the $A$-labeled hyperedge from $H^{X;Y}_{\Gamma;B}$, the node $x$ becomes isolated. Thus, $x \notin \FVar(\Gamma^\prime)$ because otherwise $x$ would be an attachment node of a hyperedge corresponding to a formula from $\Gamma^\prime$.
\end{itemize}
First, consider the hypergraph $H_1 = H^{X\cup\FVar(A);Y}_{\Gamma^\prime;B}$. It is obtained from $H_0 = H^{X;Y}_{\Gamma;B}$ by removing an $A$-labeled hyperedge and changing nothing else. Nodes from $\FVar(A) \setminus \type(H)$ are isolated in $H_1$, and $\type(H_1) = \type(H_0)$. Secondly, we claim that $H_1$ is equal to $H_2 = H^{X \cup \FVar(A);Y}_{\Gamma^\prime;A \limpl B}$. Indeed, $\type(H_2) \setminus \type(H_1) \subseteq (\FVar(\Gamma^\prime) \cup X \cup \FVar(A)) \cap \FVar(A) = \FVar(A)$, but $\type(H_1) = \type(H_0)\supseteq \FVar(A)$, so $\type(H_2) \setminus \type(H_1) = \emptyset$ and thus $\type(H_1) = \type(H_2)$. Applying Lemma \ref{lemma:properties-HXYGA} shows that $H_1=H_2$.

Finally, the second statement of Lemma \ref{lemma:properties-HXYGA} implies that there are $X^\prime,Y^\prime$ such that $H^{X^\prime;Y^\prime}_{\Gamma^\prime;A \limpl B}$ is obtained from $H_2$ by removing nodes that belong to $\FVar(A) \setminus \type(H)$. Therefore, $H^{X^\prime;Y^\prime}_{\Gamma^\prime;A \limpl B}$ is isomorphic to $H$ and thus $H \in \mathrm{u}(A \limpl B)$. 

Summing up, we have proved that $\langle \mathrm{T},\mathrm{u} \rangle$ is a hypergraph language model.


\end{document}